\documentclass[12pt,a4paper]{amsart}

\usepackage[utf8]{inputenc}
\usepackage[T1]{fontenc}
\usepackage[english]{babel}
\usepackage{latexsym}
\usepackage{amsmath,amssymb,amsthm, amsfonts}
\usepackage{txfonts}						
\usepackage{stmaryrd}
\usepackage[normalem]{ulem}
\usepackage{cancel}
\usepackage{color}
\usepackage{mathtools} 
\usepackage[left=2cm, right=2cm, bottom=3cm]{geometry} 

\newtheorem{Theorem}{Theorem}[section]
\newtheorem{Definition}[Theorem]{Definition}  
    
\newtheorem{Lemma}[Theorem]{Lemma}	
\newtheorem{Corollary}[Theorem]{Corollary}
\newtheorem{Remark}[Theorem]{Remark}
                  
\numberwithin{equation}{section}

\newcommand{\C}{\mathbb{C}} 
\newcommand{\R}{\mathbb{R}} 
\newcommand{\Q}{\mathbb{Q}} 
\newcommand{\Z}{\mathbb{Z}} 
\newcommand{\N}{\mathbb{N}} 
\newcommand{\T}{\mathbb{T}}

\newcommand{\abs}[1]{\left\vert #1 \right\vert}

\newcommand{\paren}[1]{\left( #1 \right)}
\newcommand{\bra}[1]{\left[ #1 \right]}

\newcommand{\ang}[1]{\left\langle #1 \right\rangle}
\newcommand{\mc}[1]{\mathcal{#1}}

\newcommand{\DEL}[1]{}

\allowdisplaybreaks

\makeatletter
\newcommand{\ps@bw}{\ps@empty%
  \renewcommand{\@oddfoot}%
  {\kuerzel \hfil {\footnotesize --- bitte wenden ---}}}
\newcommand{\ps@last}{\ps@empty%
  \renewcommand{\@oddfoot}{\kuerzel \hfil {\tiny\texturl}}}
\newcommand{\itemii}[1]{\refstepcounter{enumi}\hfill
  \hbox to 0.5\textwidth {\hskip \leftmargin \hskip -\labelwidth
    \hskip -\labelsep \hbox to\labelwidth {\makelabel{\@itemlabel}}%
    \hskip \labelsep #1\hfil}}
\newcommand{\itemiii}[1]{\refstepcounter{enumii}\hfill
  \hbox to 0.5\textwidth {\hskip \leftmargin \hskip -\labelwidth
    \hskip -\labelsep \hbox to\labelwidth {\makelabel{\@itemlabel}}%
    \hskip \labelsep #1\hfil}}
\makeatother

\DeclareMathOperator{\supp}{supp}
\DeclareMathOperator{\loc}{loc}

\DeclareMathOperator{\const}{const.}

\DeclareMathOperator{\essinf}{ess\ inf}

\DeclareMathOperator{\odd}{odd}

\DeclareMathOperator{\per}{per}

\DeclareMathOperator{\mono}{mono}

\newtagform{pm}{(}{)\rlap{$_\pm$}}

\begin{document}

\makeatother
\makeatletter
\makeatother


\title[Real-valued, time-periodic localized weak solutions]{Real-valued, time-periodic localized weak solutions for a semilinear wave equation with periodic potentials}

\date{\today}

\author{Andreas Hirsch}
\address{A. Hirsch \hfill\break 
Institute for Analysis, Karlsruhe Institute of Technology (KIT), \hfill\break
D-76128 Karlsruhe, Germany}
\email{andreas.hirsch@kit.edu}

\author{Wolfgang Reichel}
\address{W. Reichel \hfill\break 
Institute for Analysis, Karlsruhe Institute of Technology (KIT), \hfill\break
D-76128 Karlsruhe, Germany}
\email{wolfgang.reichel@kit.edu}
  
\subjclass[2000]{Primary: 35L71, 49J40; Secondary: 35B10}

\keywords{semilinear wave equation, breather solutions, time-periodic, variational methods}

\begin{abstract}
We consider the semilinear wave equation $V(x) u_{tt} -u_{xx}+q(x)u = \pm f(x,u)$ for three different classes (P1), (P2), (P3) of periodic potentials $V,q$. (P1) consists of periodically extended delta-distributions, (P2) of periodic step potentials and (P3) contains certain periodic potentials $V,q\in H^r_{\per}(\R)$ for $r\in [1,3/2)$. Among other assumptions we suppose that $|f(x,s)|\leq c(1+ |s|^p)$ for some $c>0$ and $p>1$. In each class we can find suitable potentials that give rise to a critical exponent $p^\ast$ such that for $p\in (1,p^\ast)$ both in the ``+'' and the ``-'' case we can use variational methods to prove existence of time-periodic real-valued solutions that are localized in the space direction. The potentials are constructed explicitely in class (P1) and (P2) and are found by a recent result from inverse spectral theory in class (P3). The critical exponent $p^\ast$ depends on the regularity of $V, q$. Our result builds upon a Fourier expansion of the solution and a detailed analysis of the spectrum of the wave operator. In fact, it turns out that by a careful choice of the potentials and the spatial and temporal periods, the spectrum of the wave operator $V(x)\partial_t^2-\partial_x^2+q(x)$ (considered on suitable space of time-periodic functions) is bounded away from $0$. This allows to find weak solutions as critical points of a functional on a suitable Hilbert space and to apply tools for strongly indefinite variational problems.
\end{abstract}


\maketitle

\section{Introduction and results}

We study the $1+1$ dimensional semilinear wave equation
\begin{equation} \label{Einl2scalar} \usetagform{pm}
V(x) u_{tt} -u_{xx} + q(x) u = \pm f(x,u) \text{ in } \R\times\R
\end{equation}
both for the plus and the minus case. Here $V,q \geq 0$ with $q(x)=\tau\omega^2 V(x)$ for $0 \leq |\tau|<\tau_0$ are periodically distributed potentials belonging to one of the three classes (P1), (P2), (P3) given below. Moreover $f:\R\times\R\to\R$ is a Carath\'{e}odory function growing at infinity with a power at most $p>1$ where $p\in (1,p^\ast)$ belong to a subcritical range of exponents, cf. the detailed assumptions (H1)--(H4) on $f$. A typical example is $f(x,s)=\Gamma(x)|s|^{p-1} s$ with a $2\pi$-periodic continuous function $\Gamma$, $\min_\R\Gamma>0$ and $1<p<p^\ast$. We are looking for real-valued, time-periodic and spatially localized solutions of \eqref{Einl2scalar}$_\pm$ often called breathers. Equation \eqref{Einl2scalar}$_\pm$ is a prototype semilinear wave equation which, e.g., can be viewed as an approximation of a second-order in time Maxwell equation for the polarized electric field in the presence of nonlinearities, cf. \cite{BlaSchneiChiril}. 
Our result is motivated by the work of Blank, Chirilus-Bruckner, Lescarret, Schneider \cite{BlaSchneiChiril} who considered \eqref{Einl2scalar} with $f(x,s)=s^3$. For a very specific choice of periodic step-functions $V$ and $q$ they proved the existence of breathers with the help of spatial dynamics, bifurcation theory and center manifold theory.

\medskip

The use of variational tools is the main methodical difference of our paper to \cite{BlaSchneiChiril}. One of the advantages of variational methods is that they allow nonlinearities which are more general than a pure power as in \cite{BlaSchneiChiril}. Further differences and advantages to \cite{BlaSchneiChiril} are pointed out in Remark~\ref{rem_nach_Hauptres} below. In the present paper we extend the results of \cite{BlaSchneiChiril} and consider the following three classes of more general potentials:
\begin{itemize}
\item[(P1)] $V(x)= \alpha +\beta \delta_{\per}(x)$, where $\alpha,\beta>0$ and $\delta_{\per}$ is the $2\pi$-periodic extension of the delta distribution supported w.l.o.g. on the set $\{2n\pi: n\in \Z\}$.
\item[(P2)] $V(x) = \alpha \chi^{\per}_{[0,2\pi\theta]}+ \beta(1-\chi^{\per}_{[0,2\pi\theta]})$, where $\alpha, \beta>0$, $\theta\in (0,1)$ and $\chi^{\per}_{[0,2\pi\theta]}$ is the $2\pi$-periodic extension of the characteristic function on $[0,2\pi\theta]$.
\item[(P3)] $V\in H^r_{\per}(\R)$ for $r\in [1,3/2)$. 
\end{itemize}
Here $H^r_{\per}(\R)$ consists of $2\pi$-periodic functions (or distributions) $V$ with Fourier-coefficients $\hat V(n) := \frac{1}{2\pi}\int_0^{2\pi} V(x) e^{-inx}\,dx$ satisfying  $\|V\|_{H^r}  := \left(\sum_{n\in\Z} |\hat V(n)|^2 (1+n^2)^r\right)^{1/2}<\infty$. Notice that if $V$ belongs to (P1) then $V\in H^r_{\per}(\R)$ for all $r<-\frac{1}{2}$ and if $V$ belongs to (P2) then $V\in H^r_{\per}(\R)$ for all $r<\frac{1}{2}$. 

\medskip

As we shall see in the main result of Theorem~\ref{Hauptresultat} each of the three classes (P1), (P2), (P3) gives rise to a critical exponent $p^\ast>1$ that limits the maximal growth of the nonlinearity $f$ in the right-hand side of \eqref{Einl2scalar}$_\pm$. Our conditions on $f$ are  the following:
\begin{itemize}
\item[(H1)] $f\colon \R\times\R\to\R$ is a continuous function which is $2\pi$-periodic in the first variable with $\vert f(x,s)\vert \leq c(1+\vert s\vert^p)$ for some $c>0$ and $p>1$,
\item[(H2)] $f(x,s)=o(s)$ as $s\to 0$ uniformly in $x\in\R$,
\item[(H3)] $f(x,s)$ is odd in $s\in\R$ and $s \mapsto f(x,s)/|s|$ is strictly increasing on $(-\infty,0)$ and $(0,\infty)$,
\item[(H4)] $\frac{F(x,s)}{s^2}\to\infty$ as $s\to\infty$ uniformly in $x\in \R$,
\end{itemize}
where $F(x,s)\coloneqq \int_0^s f(x,t) dt$.

\medskip

Due to the required $T$-periodicity in time of the solution we consider a polychromatic, real-valued solution ansatz
\begin{align}\label{AnsPolychrom} 
u(x,t)=\sum_{k\in 2\Z+1} u_k(x) e^{ik\omega t},\quad \bar u_k(x) = u_{-k}(x), \quad \omega= \frac{2\pi}{T}.
\end{align}
Inserting this ansatz into the wave operator $L_{x,t}=V(x)\partial_t^2-\partial_x^2+q(x)$ and recalling that $q(x)=\tau\omega^2 V(x)$ with $|\tau|<\tau_0$ one naturally finds the self-adjoint elliptic operator $L_k: D(L_k)\subset L^2(\R) \to L^2(\R)$ defined by 
\begin{equation}
L_k := -\frac{d^2}{dx^2} -\omega^2(k^2-\tau)V(x).
\label{definition_Lk}
\end{equation}
The way the potentials $V, q$ and the frequency $\omega$ are constructed leads to $L_k$ having a spectral gap $(-c|k|^\gamma, c|k|^\gamma)$ around 0 which grows with order $\gamma$ in $|k|$, cf. Lemma~\ref{nullinaufgehenderluecke2}, Lemma~\ref{nullinaufgehenderluecke_step} and Lemma~\ref{ex_of_potential} in Section~\ref{Deltapointinteract}. This spectral gap growing in $|k|$ is the key to finding breathers as critical points of a strongly indefinite functional by variational methods.

\medskip

The function $u$ generated by the Fourier decomposition \eqref{AnsPolychrom} is $T$-periodic in time and real-valued due to the assumption $u_k(x) = \bar u_{-k}(x)$. Since we only consider coefficients with odd indices $k\in 2\Z+1$ the function $u$ is in fact $T/2$-antiperiodic. The space of antiperiodic-in-time functions is important since it prevents the $k=0$-mode and thus keeps $0$ out of the spectrum of the wave operator $L_{x,t}=V(x)\partial_t^2-\partial_x^2+ q(x)$. At the same time by (H3) the nonlinearity $f(x,u)$ is odd in the second variable and hence it is consistent with seeking $T/2$-antiperiodic solutions. 

\medskip

The space-time domain on which the solutions are determined is denoted by $D\coloneqq \R\times (0,T)$. The potentials $V$ belonging to (P2) and (P3) are bounded, and hence the concept of weak solutions for \eqref{Einl2scalar} given next would only require $u\in L^2(D)\cap L^{p+1}(D)$. However, if $V$ belongs to (P1) constructed from a $2\pi$-periodic extension of the $\delta$-distribution, then we need a suitable adaptation of the concept of a weak solution. Let $\T$ stand for the one dimensional flat $2\pi$-periodic torus. For $r,s \in \R$ we denote by $H^r(0,T; H^1(\R))$, $H^s(0,T; L^2(\R))$ the Bochner spaces of functions of the type \eqref{AnsPolychrom} with the respective norms 
$$
\|u\|_{H^r(0,T;H^1)}^2 = \sum_{k\in 2\Z+1} |k|^{2r} \|u_k'\|_{L^2}^2 \quad \mbox{ and } \quad \|u\|_{H^s(0,T;L^2)}^2 = \sum_{k\in 2\Z+1} |k|^{2s} \|u_k\|_{L^2}^2.
$$

\begin{Definition} \label{Defveryweaksol1} Let $V$ belong to one of the classes (P1), (P2), (P3) and let $q(x) = \tau\omega^2 V(x)$ for some $\tau\in \R$.  We call $u$ of the form \eqref{AnsPolychrom} with $u\in L^{p+1}(D)$ and $u\in H^r(0,T; H^1(\R))\cap H^s(0,T; L^2(\R))$ for some $r\in \R, s>0$ a weak $T$-periodic solution of \eqref{Einl2scalar}$_\pm$ if 
\begin{equation}  \label{Defweaksolnext}
 \int_D V(x) u \phi_{tt} - u \phi_{xx} + q(x) u \phi\,d(x,t) = \pm\int_D f(x,u) \phi\, d(x,t)
\end{equation}
holds for every $\phi\in C_c^\infty(\R\times\T)$. In the case of class (P1) the above notation is understood as 
\begin{equation} \label{specification}
\int_D \delta_{\per}(x) u \phi \,d(x,t) = \sum_{n\in\Z} \int_0^T u(2\pi n,t)\phi(2\pi n,t)\,dt
\end{equation}
for all $\phi \in C_c^\infty(\R\times\T)$. 
\end{Definition}

\begin{Remark} Since $u\in  H^s(0,T; L^2(\R))$ for some $s\geq 0$ we have $u\in L^2(D)$. Hence all of the above integrals are well defined in the cases where $V$ belongs to (P2) or (P3) since then $V\in L^\infty(\R)$. In the case where $V$ belong to (P1) we have $u \in H^r(0,T; C(\R))$ and hence $\int_0^T u(2\pi n,t)\phi(2\pi n,t)\,dt$ is well-defined for every test function $\phi\in C_c^\infty(\R\times\T)$. Notice that in this case the sum in the right-hand side of \eqref{specification} is finite.
\end{Remark}

Based on this concept of a weak solution our main result reads as follows.
\begin{Theorem} \label{Hauptresultat}
Let $V$ be a potential which together with $\omega$, $\tau_0$ and $p^\ast$ satisfies one of the following three assumptions (V1), (V2), (V3). Moreover, let  $f$ satisfy (H1)--(H4) with $p\in (1,p^\ast)$ and let $q(x) = \tau\omega^2 V(x)$ with $0 \leq |\tau|<\tau_0$. Then \eqref{Einl2scalar}$_\pm$ possesses a non-trivial $\frac{2\pi}{\omega}$-periodic weak solution in the sense of Definition~\ref{Defveryweaksol1}. The assumptions on $V$ are as follows:
\begin{itemize}
\item[(V1)] $V$ belongs to (P1) with $\alpha>0, \beta>32\alpha$, $\tau_0 = 1- \frac{32\alpha}{\beta}$, $\omega = \frac{1}{4\sqrt{\alpha}}$ and $p^\ast=2$.
\item[(V2)] $V$ belongs to (P2) with $\theta$ or $1-\theta$ belonging to $(0,\frac{1}{2}(1-\sqrt{7/9}))$, $\alpha>0$, $\theta^2\alpha = (1-\theta)^2\beta$, $\tau_0 = 1- \frac{16\theta(1-\theta)}{\theta^2+(1-\theta)^2}$, $\omega = \frac{1}{4\theta \sqrt{\alpha}}= \frac{1}{4(1-\theta)\sqrt{\beta}}$ and $p^\ast=3$.
\item[(V3)] Let $r\in [1,\frac{3}{2})$ and $0<\gamma<\frac{3}{2}-r$. $V$ belongs to (P3) and is chosen together with $\tau_0$ according to Lemma~\ref{ex_of_potential} below with $\omega = \pi/\int_0^{2\pi} \sqrt{V(x)}\,dx$ and $p^\ast = \frac{2+\gamma}{2-\gamma}$.
\end{itemize}
\end{Theorem}

\begin{Remark} \label{rem_nach_Hauptres} (i) In contrast to \cite{BlaSchneiChiril} $\tau=0$, i.e., $q\equiv0$ is admissible. Also, in contrast to \cite{BlaSchneiChiril} our breathers are not small since they do not arise as local bifurcations from the trivial solution. \\ 
(ii) Since $f(x,s)$ is odd in $s$ one can expect the existence of infinitely many breathers as critical points of the functional $J$ in Section~\ref{MinimgenNM}. The arguments needed to establish infinitely many critical points for an indefinite functional are generally known, e.g., cf. Theorem~1.2 in \cite{SW_inf}, but would go beyond the scope of the present paper.\\
(iii) Notice that in assumption (V3) $p^\ast$ can reach any value in $(1,\frac{7-2r}{1+2r})$ since $\gamma$ can be chosen arbitrarily close to $\frac{3}{2}-r$.\\
(iv) In the theorem we do not cover the case $p=p^\ast$. The reason is that $p^\ast+1$ has the character of a critical Sobolev-exponent, cf. Remark~\ref{crit_sob}. Since local compactness properties of certain embeddings are lost for endpoint cases, additional difficulties arise in the case $p=p^\ast$ that would substantially extend the length of the present paper. In \cite{BlaSchneiChiril} the case $p=3=p^\ast$ in case (V2) is included. We attempt to address the endpoint case $p=p^\ast$ in a subsequent paper.
\end{Remark}

The regularity of the solution from Theorem~\ref{Hauptresultat} is given in detail next.

\begin{Corollary} \label{CorzuHauptresultat}
Under the assumptions of Theorem~\ref{Hauptresultat} the solution $u$ satisfies $u\in H^\alpha(0,T; H^1(\R))\cap H^\beta(0,T; L^2(\R))$ with 
\begin{equation*}
 2\alpha= \left\{\begin{array}{ll}
 -3 & \mbox{ in case (V1)}, \vspace{\jot} \\
 -1 & \mbox{ in case (V2)}, \vspace{\jot} \\
 \gamma-2 & \mbox{ in case (V3)}, 
\end{array}
\right.
\qquad 
2\beta= \left\{\begin{array}{ll}
 1 & \mbox{ in case (V1)}, \vspace{\jot} \\
 1 & \mbox{ in case (V2)}, \vspace{\jot} \\
 \gamma & \mbox{ in case (V3)}. 
\end{array}
\right.
\end{equation*}
We can therefore weaken the assumptions on the test functions $\phi$ in Definition~\ref{Defveryweaksol1}: we can replace $- \int_D u \phi_{xx} \, d(x,t)$ by $\int_D u_x\phi_x\,d(x,t)$ and admit $\phi\in H^{\tilde \alpha}(0,T;H^1(\R))\cap H^{\tilde\beta}(0,T;L^2(\R))$, where 
\begin{equation*}
 2\tilde \alpha\geq  \left\{\begin{array}{ll}
 3 & \mbox{ in case (V1)}, \vspace{\jot} \\
 1 & \mbox{ in case (V2)}, \vspace{\jot} \\
 2-\gamma & \mbox{ in case (V3)}, 
\end{array}
\right.
\qquad 
2\tilde \beta \geq  \left\{\begin{array}{ll}
 5 & \mbox{ in case (V1)}, \vspace{\jot} \\
 3 & \mbox{ in case (V2)}, \vspace{\jot} \\
 4-\gamma & \mbox{ in case (V3)} 
\end{array}
\right.
\end{equation*}
and additionally $\tilde\alpha+\tilde\beta \geq 5$ in case (V1).
\end{Corollary}
 
\smallskip

Breather solutions of nonlinear wave equations are quite rare. After the discovery of the sine-Gordon breather family, cf. \cite{ab_kaup_newell_segur:73} 
\begin{align*}
u_{m,\omega}(x,t)=4\arctan \paren{\frac{m}{\omega} \frac{\sin(\omega t)}{\cosh (mx)}}, m,\omega>0, m^2+\omega^2=1
\end{align*} 
for the sine-Gordon equation
\begin{align} \label{SineGordonequation} 
u_{tt}-u_{xx}+\sin u = 0 \text{ in } \R\times\R
\end{align}
a number of results on the non-existence of breathers appeared, e.g. \cite{seguar_kruskal}, \cite{Birnir}, \cite{Denzler}, and most recently in \cite{kowalczyk_et_al}. By these works it became clear that breathers do not persist in homogeneous nonlinear wave equations if the $\sin u$ nonlinearity in \eqref{SineGordonequation} is perturbed to $f(u)$ with $f(0)=0, f'(0)>0$. Thus, the existence of breathers in nonlinear wave equations like $u_{tt}-u_{xx}+f(u)=0$ is a rare phenomenon. The situation is different if one introduces inhomogeneities. For example, nonlinear wave equations on discrete lattices can support breather solutions, cf. \cite{mackay_aubry} for a fundamental result and \cite{james_breathers:09} for an overview with many references. Another way to recover breathers is to introduce inhomogeneities via $x$-dependent coefficients like in \cite{BlaSchneiChiril} for \eqref{Einl2scalar} with $f(x,s)=s^3$. Recently, the authors in \cite{PlumReichel} gave an existence result for breathers in the $3+1$-dimensional semilinear curl-curl wave equation
\begin{align*}
V(x)\partial_t^2 U + \nabla\times\nabla\times U+q(x) U\pm \Gamma(x)\vert U\vert^{p-1} U=0,\ \  p>1,
\end{align*}
for radially symmetric, positive and non-constant functions $V,q,\Gamma\colon \R^3\to (0,\infty)$ satisfying further properties not listed here (note that in \cite{PlumReichel} instead of $V,q,\Gamma$ the potentials are called $s, q, V$). Another interesting polychromatic approach for finding coherent spatially localized solutions of the 1+1-dimensional (quasilinear) Maxwell model is given in \cite{PelSimWeinstein}. Based on a multiple scale ansatz the field profile is expanded into infinitely many modes which are time-periodic both in the fast and slow time variables. Since the periodicities in the fast and slow time-variables differ, the field becomes quasiperiodic in time. The resulting system for these infinitely many coupled modes is to a certain extent treated analytically, with a rigorous existence proof yet missing. The numerical results of \cite{PelSimWeinstein} indicate that spatially localized solitary waves could exist, although nonexistence has not yet been ruled out.

\smallskip

Our main tool for proving existence of breather solutions for \eqref{Einl2scalar} is the use of variational methods. In the context of semilinear wave equations with Dirichlet boundary value problems on intervals of length $\pi$ variational methods have been used before to show existence of time-periodic solutions. E.g. \cite{brezis_coron_nirenberg}, \cite{brezis_coron} used dual variational techniques to prove the existence of $T$-periodic solutions for $u_{tt}-u_{xx}+g(u)=0$ for monotone increasing nonlinearities $g:\R\to\R$ provided $T/\pi\in \Q$. In \cite{hofer} tools for the existence of critical points of strongly indefinite functionals associated to semilinear wave equations on intervals of length $\pi$ are exploited. These variational approaches build on the fact that in the space of $T$-periodic functions the operator  $\partial_t^2-\partial_x^2$ has discrete spectrum due to the Dirichlet boundary conditions. This fails for \eqref{Einl2scalar} because of the unbounded spatial domain $\R$. Yet another aspect of variational methods applied to semilinar wave equations appeared recently in \cite{alejo_et_al_stability}: there the authors study the stability of the sine-Gordon breather using its variational structure together with spectral assumptions on the linearized operator for which strong numerical evidence is given.

\medskip

The paper is structured as follows: In the next section we construct examples of potentials $V,q$ according to (V1), (V2), (V3) which lead to a spectral gap of $L_k$ around $0$ which grows in $|k|$. These and further properties of the operator $L_k$ are described in Section~\ref{Sec:prop_Lk}. The functional analytic framework for breathers is given in Section~\ref{Sec:fa} via a suitable Hilbert-space $\mc{H}$ for the temporal Fourier-coefficients. An important part is the integrability properties of functions composed from these temporal Fourier-coefficients as described in Theorem~\ref{HauptresultatSection5.5}. Because the proof of this theorem is rather long, we have moved it to Section~\ref{sec:proof_for_S}. The use of the integrability properties allows to incorporate nonlinearities into the variational setting. In Section~\ref{MinimgenNM} we find minimizers of a suitable functional on the so-called generalized Nehari manifold, and show that they give rise to weak solutions of \eqref{Einl2scalar}$_\pm$ with regularity properties as given in Corollary~\ref{CorzuHauptresultat}. In order to keep the main sections non-technical, some technical aspects (e.g. a concentration-compactness Lemma) are shifted to the appendix. Throughout this paper we write $\Z_{\odd}\coloneqq 2\Z+1$.

\section{Spectral analysis for examples of one-dimensional operator families} \label{Deltapointinteract}

We consider the one-dimensional family of elliptic operators $L_k: D(L_k)\subset L^2(\R) \to L^2(\R)$ given by 
$$
L_k := -\frac{d^2}{dx^2} - k^2\omega^2 V(x) + q(x), \quad k \in \Z_{\odd}.
$$
We construct examples of $2\pi$-periodic potentials $V,q$ so that $L_k$ has a spectral gap around $0$ of the size $\const |k|^\gamma$ for certain values of $\gamma$ depending on the the cases (V1), (V2), (V3). Consider the closed and semibounded bilinear form
$$
b_{L_k}(u_k,v_k) = \int_\R u_k' \bar v_k'\,dx +\int_\R \Bigl(-k^2\omega^2 V(x)+q(x)\Bigr) u_k \bar v_k\,dx, \quad u_k, v_k \in D(b_{L_k})=H^1(\R),
$$
where in case (V1) we interpret $\int_\R \delta_{\per}(x) u_k \bar v_k = \sum_{n\in\Z} u_k(2\pi n)\bar v_k(2\pi n)$. By Theorem~VIII.15 in \cite{ReedSimon1} we may view $L_k$ as a self-adjoint operator on a suitable domain $D(L_k)$ given by the relation $\langle L_k \phi,\psi\rangle_{L^2(\R)} = b_{L_k}(\phi,\psi)$ for all $\phi\in D(L_k)$ and all $\psi\in H^1(\R)$. The spectrum and the resolvent set of $L_k$ will be denoted by $\sigma(L_k), \rho(L_k)$, respectively. Due to the periodicity of the potentials $V,q$ the spectrum of $L_k$ has band-gap structure which will be analyzed in detail in the following three sections. 

\medskip

The following lemma turns out to be useful for the subsequent computations.

\begin{Lemma} \label{help} Let $\tau\in (-1,1)$. Then there exists $c>0$ such that 
$$
\left|\sin\left(2\pi\sqrt{\lambda+\frac{1}{16}(k^2-\tau)}\right)\right| \geq \min\left\{\frac{1}{2}\sqrt{1+\tau}, \frac{1}{2}\sqrt{1-\tau}\right\} \mbox{ for all } k\in \N_{\odd}, \lambda\in(-ck,ck).
$$
\end{Lemma}

\begin{proof} Consider $\lambda\in(-ck,ck)$. Let $\delta = 1- \frac{1}{2}\sqrt{1-\tau}$. Due to $\tau\in (-1,1)$ we have $\delta\in (0,1)$ and moreover $\tau+\delta^2-2\delta=-\frac{3}{4}(1-\tau)<0$. Then we choose $c>0$ so small that $-16c-\tau > \delta^2-2\delta$ and $-16c>-2\delta$. Consequently, for all $k\in \N$ we obtain $16\lambda +k^2-\tau > -16kc + k^2-\tau = k(-16c+2\delta)-2k\delta+k^2-\tau\geq -16c+2\delta-2k\delta+k^2-\tau>(k-\delta)^2$ and hence
\begin{equation} \label{lower}
2\pi\sqrt{\lambda+\frac{1}{16}(k^2-\tau)}> \frac{\pi}{2}(k-\delta).
\end{equation}
Similarly, let $\epsilon= 1-\frac{1}{2}\sqrt{1+\tau}$. Due to $\tau\in (-1,1)$ we have $\epsilon\in (0,1)$ and $\tau+\epsilon^2+2\epsilon>\tau+2\epsilon>0$. Then (by possibly decreasing $c>0$) we may assume $16c-\tau < \epsilon^2+2\epsilon$ and $16c < 2\epsilon$. Thus, for all $k\in \N$ we obtain $16\lambda+k^2-\tau < 16kc+k^2-\tau =k(16c-2\epsilon)+2k\epsilon+k^2-\tau < 16c-2\epsilon +2k\epsilon +k^2-\tau< (k+\epsilon)^2$ and hence 
\begin{equation} \label{upper}
2\pi\sqrt{\lambda+\frac{1}{16}(k^2-\tau)}< \frac{\pi}{2}(k+\epsilon).
\end{equation}
Combining \eqref{lower}, \eqref{upper}, $\epsilon, \delta\in (0,1)$ and $k\in \N_{\odd}$ we get that 
$$
\left|\sin\left(2\pi\sqrt{\lambda+\frac{1}{16}(k^2-\tau)}\right)\right| \geq \min \left\{ \cos\Bigl(\frac{\pi\epsilon}{2}\Bigr), \cos\Bigl(\frac{\pi\delta}{2}\Bigr)\right\} \geq \min\{1-\epsilon, 1-\delta\}
$$
which yields the statement of the lemma.
\end{proof}

\subsection{Periodic delta potential} \label{PerdeltaBeispiel}
We consider first the one-dimensional differential expression
\begin{align} \label{1dperdelta}
Lu \coloneqq -u''+ (\tilde{\alpha}+\tilde{\beta} \delta_{\per}(x))u \text{ on } \R,
\end{align}
where $\tilde{\alpha}\in \R$ and $\tilde{\beta} \in \R\setminus \{0\}$. We always assume that $\delta_{\per}$ is supported on $I_\delta:=\{2n\pi: n\in\Z\}$, is $2\pi$-periodic and acts as a delta-distribution at each of the points $2n\pi$ for $n\in \Z$. By Theorem 1 in \cite{ChristStolz} the operator $L$ in \eqref{1dperdelta} is self-adjoint on the domain
\begin{align} \label{DefLStolzChrist}
\begin{split}
D(L)&\coloneqq \Big\{ u\in L^2(\R): u \text{ abs. cont. on } \R, u' \text{ abs. cont. on } \R\setminus I_\delta, \Big. \\
&\Big. u'(x_+)-u'(x_-)=\tilde{\beta} u(x) \text{ for all } x\in I_\delta \text{ and } -u''+ \tilde{\alpha} u \in L^2(\R)\Big\}.
\end{split}
\end{align}
In \eqref{DefLStolzChrist} the function $u$ is continuous on $\R$ and $u', u''$ exist pointwise almost everywhere and are $L^2$-integrable. We rewrite the domain of definition in \eqref{DefLStolzChrist} by making use of weak derivatives. In the following $u$ is a continuous $L^2$-function with an $L^2$-integrable weak derivative $u'$, whereas $u''$ is not a function anymore but a distribution. Thus,
\begin{align*}
D(L)=\{&u\in L^2(\R): Lu\in L^2(\R)\} = \Bigl\{u\in H^1(\R), u|_{(2\pi n, 2\pi (n+1))}\in H^2(2\pi n,2\pi (n+1)) \Bigr.\\ 
&\Bigl.\text{ for all } n\in\Z, \sum_{n\in\Z} \Vert u''\Vert_{L^2(2\pi n,2\pi(n+1))}^2<\infty, u'(x_+)-u'(x_-)=\tilde{\beta} u(x) \text{ for all } x\in I_\delta\Bigr\}.
\end{align*}

\DEL{
{\color{red} \texttt{das k\"onnen wir wohl l\"oschen}
We now introduce the concept of a weak solution of $Lu=f$. 

\begin{Definition}
For $f\in L^2(\R)$ we say that $u\in H^1(\R)$ is a weak solution of $Lu=f$  with $L$ as in \eqref{1dperdelta} if
\begin{align*}
\int_{\R} \paren{u'(x)\varphi'(x)+\tilde{\alpha}u(x)\varphi(x)} dx+\tilde{\beta}\sum_{n\in\Z} u(2\pi n)\varphi(2\pi n) = \int_{\R} f(x)\varphi(x) dx
\end{align*}
holds true for all $\varphi\in C_c^\infty(\R)$. Furthermore, for $u,v\in H^1(\R)$ we define the bilinear form $b$ associated to $L$ by
\begin{align} \label{BilzudeltaDef}
b(u,v) := \int_{-\infty}^\infty \paren{u'(x)\overline{v'(x)} + \tilde{\alpha} u(x)\overline{v(x)}} dx + \tilde{\beta} \sum_{n\in \Z} u(2\pi n)\overline{v(2\pi n)}. 	
\end{align}
\end{Definition}
The bilinear form $b$ and operator $L$ are related via $b(u,v)=\ang{Lu,v}_{L^2(\R)} \text{ for all } u\in D(L),v\in H^1(\R)$, see Theorem~VIII.15 in \cite{ReedSimon1}. 
}}

In \cite{BuschmannStolzWeidmann} it is shown that the classical Sturm-Liouville theory can be generalized to include delta-point interactions, see also the appendix of \cite{ChristStolz}. One can describe the spectrum of $L$ by using the so-called discriminant $D$ (compare Chapter~1 and §~2.1 in \cite{EasthamPer}). Here the discriminant is defined as follows: for $\lambda\in \R$ let $v_1, v_2:\R\to\R$ be solutions of $L v_i = \lambda v_i$ with initial conditions $v_1(x_0)=1, v_1'(x_0)=0$ and $v_2(x_0)=0, v_2'(x_0)=1$ for some $x_0\not\in I_\delta$. Then $v_1, v_2$ is a fundamental system of solutions for the equation $Lu=\lambda u$ and the discriminant is defined as 
$$
D(\lambda) := v_1(x_0+2\pi)+v_2'(x_0+2\pi).
$$
Following Chapter~1 and §~2.1 in \cite{EasthamPer} the spectrum $\sigma(L)$ is characterized with the help of $D(\lambda)$.

\begin{Theorem} \label{Charspec}
$\sigma(L)=\{\lambda\in\R: \abs{D(\lambda)}\leq 2\}$.
\end{Theorem}

Next we present the exact form of $D$ associated to \eqref{1dperdelta}. The proof is a straightforward computation so we omit it.
\begin{Lemma} \label{1ddeltaallg}
The discriminant $D(\cdot)$ associated to \eqref{1dperdelta} reads
\begin{align} \label{deltaformelallg}
D(\lambda) = \begin{cases} \displaystyle\frac{\tilde{\beta}\sin(2\pi \sqrt{\lambda- \tilde{\alpha}})}{\sqrt{\lambda-\tilde{\alpha}}}  +2\cos(2\pi \sqrt{\lambda -\tilde{\alpha}}) &\text{ for } \lambda-\tilde{\alpha} > 0, \vspace{\jot}\\ 
\hspace{3cm} 2+2\pi\tilde{\beta} &\text{ for } \lambda-\tilde{\alpha}=0, \vspace{\jot}\\
\displaystyle\frac{\tilde{\beta}\sinh(2\pi \sqrt{-(\lambda- \tilde{\alpha})})}{\sqrt{-(\lambda- \tilde{\alpha})}}  +2\cosh(2\pi \sqrt{-(\lambda -\tilde{\alpha})}) &\text{ for } \lambda-\tilde{\alpha}< 0.  \end{cases}
\end{align}
\end{Lemma}

If we insert $V(x)=\alpha+\beta\delta_{\per}(x)$ into \eqref{definition_Lk} we get the following representation for the operator $L_k$, $k\in\Z_{\odd}$:
\begin{align} \label{DefinitionOpLk}
L_k = -\frac{d^2}{dx^2} -\alpha\omega^2 (k^2-\tau) - \beta\omega^2 (k^2-\tau)  \delta_{\per}(x).
\end{align}
By Lemma~\ref{1ddeltaallg} the discriminant $D_k$ associated to $L_k$ reads
\begin{align} \label{Dklambdadreifaelle}
D_k(\lambda)= \begin{cases} -\frac{\beta\omega^2(k^2-\tau)\sin\left(2\pi\sqrt{\lambda+\alpha\omega^2(k^2-\tau)}\right)}{\sqrt{\lambda+\alpha\omega^2(k^2-\tau)}} +2\cos(2\pi\sqrt{\lambda+\alpha\omega^2(k^2-\tau)}) &  \text{ for } \lambda>-\alpha\omega^2(k^2-\tau), \\
\hspace{3cm} 2-2\pi\beta\omega^2(k^2-\tau) & \text{ for } \lambda=-\alpha\omega^2(k^2-\tau), \\
-\frac{\beta\omega^2(k^2-\tau)\sinh\left(2\pi\sqrt{-\lambda-\alpha\omega^2(k^2-\tau)}\right)}{\sqrt{-\lambda-\alpha\omega^2(k^2-\tau)}} +2\cosh(2\pi\sqrt{-\lambda-\alpha\omega^2(k^2-\tau)}) &  \text{ for } \lambda<-\alpha\omega^2k^2.
\end{cases}
\end{align}
We compute $\sigma(L_k)$ depending on $k\in\Z_{\odd}$ by making use of Theorem~\ref{Charspec}. Since $k$ appears in $L_k$ only as $k^2$ we restrict to $k\in \N_{\odd}$. We give conditions on $(\omega,\alpha,\beta, \tau)\in \R^4$ s.t. zero lies uniformly in a spectral gap of $L_k$ for all $k\in \N_{\odd}$ in the following sense. 

\begin{Lemma} \label{nullinaufgehenderluecke2}
Let $(\omega,\alpha,\beta,\tau)\in\R^3_+\times \R$ satisfy
\begin{align} \label{choiceofconstants}
\alpha>0, \;\omega=\frac{1}{4\sqrt{\alpha}}, \;\beta>32\alpha \;\mbox{ and }\; 0\leq |\tau|<1- \frac{32\alpha}{\beta}. 
\end{align}
Then there is $c>0$ independent of $k\in \N_{\odd}$ such that $(-c\vert k\vert,c\vert k\vert)\subset \rho (L_k) \text{ for all } k\in \N_{\odd}$.
\end{Lemma}

\begin{proof} By Theorem~\ref{Charspec} we have to find $c>0$ such that $\abs{D_k(\lambda)}>2$ for all $\lambda\in (-ck,ck)$ and all $k\in \N_{\odd}$. We will choose $c>0$ so small that $0<c<\frac{1}{16}(1-\tau)$, since then $\lambda>-ck\geq -\frac{1}{16}(k^2-\tau)$ for all $k\in \N$ and hence we only have to deal with the first case of the case distinction in \eqref{Dklambdadreifaelle}. The result follows if we can guarantee that for all $k\in 2\N-1$:
\begin{align} \label{cossinterme}
\left\vert 2\cos\paren{2\pi\sqrt{\lambda+\frac{1}{16}(k^2-\tau)}}-\frac{\beta \omega^2 (k^2-\tau)}{\sqrt{\lambda+\frac{1}{16}(k^2-\tau)}}\sin\paren{2\pi \sqrt{\lambda+\frac{1}{16}(k^2-\tau)}}\right\vert >2 \text{ for } \abs{\lambda}< ck.
\end{align}
Since $\big\vert 2\cos\paren{2\pi\sqrt{\lambda+\frac{1}{16}(k^2-\tau)}}\big\vert\leq 2$ it is sufficient for \eqref{cossinterme} to prove 
\begin{align} \label{sinterm1}
 \left\vert\sin\paren{2\pi \sqrt{\lambda+\frac{1}{16}(k^2-\tau)}}\right\vert >4 
 \frac{\sqrt{\lambda+\frac{1}{16} (k^2-\tau)}}{\beta\omega^2(k^2-\tau)}
  \text{ for } \abs{\lambda}<ck \text{ and all } k\in \N_{\odd}.
\end{align}
By Lemma~\ref{help} we can choose $c>0$ so small that the lefthand side in \eqref{sinterm1} has the positive lower bound 
\begin{equation}
\min\left\{\frac{1}{2}\sqrt{1+\tau}, \frac{1}{2}\sqrt{1-\tau}\right\} \geq \min\left\{\frac{1}{2}\frac{1+\tau}{\sqrt{1-\tau}}, \frac{1}{2}\sqrt{1-\tau}\right\}= \frac{1}{2\sqrt{1-\tau}}\min\{1+\tau,1-\tau\}
\label{lower_bound}
\end{equation}
for all $k\in \N_{\odd}$ and all $\tau\in (-1,1)$. Let us find an upper bound for the right hand side of \eqref{sinterm1}. Clearly 
\begin{equation}
4\frac{\sqrt{\lambda+\frac{1}{16} (k^2-\tau)}}{\beta\omega^2(k^2-\tau)} \leq 4 \frac{\sqrt{\frac{ck}{k^2-\tau}+\frac{1}{16}}}{\beta \omega^2 \sqrt{k^2-\tau}}
\label{upper_bound}
\end{equation}
and $\frac{k}{k^2-\tau}= \frac{1}{k-\tau/k} \leq \frac{1}{1-\tau}\leq \frac{2}{1-\tau}$ for $\tau\in [0,1)$, $k\in \N_{\odd}$ and $\frac{k}{k^2-\tau}\leq \frac{1}{k}\leq 1 \leq  \frac{2}{1-\tau}$ for $\tau\in (-1,0]$, $k\in \N_{\odd}$. Thus the upper bound from \eqref{upper_bound} becomes 
\begin{equation}
4\frac{\sqrt{\lambda+\frac{1}{16} (k^2-\tau)}}{\beta\omega^2(k^2-\tau)} \leq 4 \frac{\sqrt{\frac{2c}{1-\tau}+\frac{1}{16}}}{\beta \omega^2 \sqrt{1-\tau}} \leq \frac{4\sqrt{\frac{2c}{1-\tau}}+1}{\beta \omega^2 \sqrt{1-\tau}}.
\label{upper_bound2}
\end{equation}
Combining \eqref{lower_bound} and \eqref{upper_bound2} and using $\omega^2 = 1/(16\alpha)$ we see that it is sufficient to have 
\begin{equation}
\left(8\sqrt{\frac{2c}{1-\tau}}+2\right)\frac{16\alpha}{\beta} < \min\{1+\tau,1-\tau\}
\label{both_bounds}
\end{equation}
A sufficiently small value of $c>0$ (depending on $\tau, \alpha, \beta$) satisfying \eqref{both_bounds} can be found provided 
$$
\frac{32\alpha}{\beta}< \min\{1+\tau,1-\tau\}
$$
i.e. $0 \leq |\tau| < 1 -\frac{32\alpha}{\beta}$. 
\end{proof} 

\DEL{
\begin{Remark} {\color{red} \texttt{Neu machen!}
Assumption \eqref{choiceofconstants} is precisely the assumption of Theorem~\ref{Hauptresultat} since \eqref{choiceofconstants} leads to $T=\frac{2\pi}{\omega}=8\pi\sqrt{\alpha}$. Moreover, notice that $\big\vert D_k\paren{-\frac{k^4}{4}-\frac{k^2}{16}}\big\vert = 2e^{-\pi k^2}<2$ and $\big\vert D_k\paren{\frac{k^4}{4}-\frac{k^2}{16}}\big\vert =2$, i.e., $\pm \frac{k^4}{4}-\frac{k^2}{16}\in \sigma (L_k)$. Hence there exist elements of $\sigma(L_k)$ to the left and to the right of $0$, i.e., $0$ lies in a true spectral gap of $L_k$.}
\end{Remark}
}

\subsection{Periodic step potential} \label{PerstepBeispiel}

Here we consider the one-dimensional differential expression
\begin{align} \label{1dperstep}
Lu \coloneqq -u''+\paren{\tilde\alpha \chi^{\per}_{[0,2\pi\theta]}+ \tilde\beta(1-\chi^{\per}_{[0,2\pi\theta]})}u \text{ on } \R,
\end{align}
where $\theta \in (0,1)$ and $\tilde{\alpha}, \tilde{\beta} \in \R$. Here $\chi^{\per}_{[0,2\pi\theta]}$ is the $2\pi$-periodic extension of the characteristic function on $[0,2\pi\theta]$. The operator $L$ is self-adjoint on the domain $D(L)=H^2(\R)$. We write $\theta'=1-\theta$. 

\medskip

As in the previous section, its spectrum is characterized by the discriminant $D$. The only difference is that that the initial condition for the fundamental system of solutions can be set at any point $x_0\in \R$. The computation of the exact form of $D$ associated to \eqref{1dperstep} is straightforward, so we omit it. \begin{Lemma} \label{1dstepallg}
The discriminant $D(\cdot)$ associated to \eqref{1dperstep} in the case $\lambda > \max\{\tilde\alpha,\tilde\beta\}$ reads 
\begin{equation} \label{stepformelallg}
D(\lambda) = - \frac{2\lambda-\tilde\alpha-\tilde\beta}{\sqrt{(\lambda-\tilde\alpha)(\lambda-\tilde\beta)}} \sin\left(\sqrt{\lambda-\tilde\alpha}2\pi\theta\right)\sin\left(\sqrt{\lambda-\tilde\beta}2\pi\theta'\right)+2\cos\left(\sqrt{\lambda-\tilde\alpha}2\pi\theta\right)\cos\left(\sqrt{\lambda-\tilde\beta}2\pi\theta'\right).
\end{equation}
\end{Lemma}

\begin{Remark}
Since the remaining case $\lambda\leq \max\{\tilde\alpha,\tilde\beta\}$ plays no role in the subsequent considerations we omit it.
\end{Remark}

If we insert $V(x)=\alpha \chi^{\per}_{[0,2\pi\theta]}+\beta(1-\chi^{\per}_{[0,2\pi\theta]})$ we get the following representation for the operator $L_k$, $k\in\Z_{\odd}$:
\begin{align} \label{DefinitionOpLkstep}
L_k \coloneqq -\frac{d^2}{dx^2} -\alpha\omega^2 (k^2-\tau)\chi^{\per}_{[0,2\pi\theta]} - \beta\omega^2 (k^2-\tau) (1-\chi^{\per}_{[0,2\pi\theta]}).
\end{align}
The discriminant $D_k$ associated to $L_k$ for $\lambda>(\tau-k^2)\omega^2 \min\{\alpha, \beta\}$ is given as in  Lemma~\ref{1dstepallg} with $\tilde\alpha = -\alpha\omega^2 (k^2-\tau)$ and $\tilde\beta = - \beta\omega^2 (k^2-\tau)$. We compute $\sigma(L_k)$ depending on $k\in\Z_{\odd}$ by making use of Theorem~\ref{Charspec}. Since $k$ appears in $L_k$ only as $k^2$ we restrict to $k\in \N_{\odd}$. We give conditions on $(\omega,\alpha,\beta,\theta,\tau)\in \R^5$ s.t. zero lies uniformly in a spectral gap of $L_k$ for all $k\in \N_{\odd}$ in the following sense. 

\begin{Lemma} \label{nullinaufgehenderluecke_step}
Let $\theta, \theta'\in (0,1)$ satisfy $\theta+\theta'=1$ and either $\theta$ or $\theta'$ belong to $(0,\frac{1}{2}(1-\sqrt{7/9}))$. Let moreover $(\omega,\alpha,\beta,\tau)\in\R_+^3\times\R$ satisfy
\begin{align} \label{choiceofconstants_step}
\alpha>0, \; \omega=\frac{1}{4\theta \sqrt{\alpha}}=\frac{1}{4\theta' \sqrt{\beta}} \;\mbox{ and }\;  0\leq |\tau|<1- \frac{16\theta \theta'}{\theta^2+{\theta'}^2} .
\end{align}
Then there is $c>0$ independent of $k\in \N_{\odd}$ such that $(-c\vert k\vert,c\vert k\vert)\subset \rho (L_k) \text{ for all } k\in \N_{\odd}$.
\end{Lemma}

\begin{proof} By Theorem~\ref{Charspec} we have to find $c>0$ such that $\abs{D_k(\lambda)}>2$ for all $\lambda\in (-ck,ck)$ and all $k\in 2\N-1$. First choose $c>0$ so small that $0<c<(1-\tau)\omega^2\min\{\alpha,\beta\}$. This implies $\lambda>-ck\geq -(k^2-\tau)\omega^2\min\{\alpha,\beta\}$ for all $k\in \N$ and therefore \eqref{stepformelallg} in Lemma~\ref{1dstepallg} gives the form of $D(\lambda)$. The result follows as in Lemma~\ref{nullinaufgehenderluecke2} if for all $k\in 2\N-1$ we have 
\begin{multline} \label{sinterme}
\sin\left(\sqrt{\lambda+\alpha\omega^2(k^2-\tau)}2\pi\theta\right)\sin\left(\sqrt{\lambda+\beta\omega^2(k^2-\tau)}2\pi\theta'\right)  \\
> 4 \frac{\sqrt{(\lambda+\alpha\omega^2(k^2-\tau))(\lambda+\beta\omega^2(k^2-\tau))}}{2\lambda+(\alpha+\beta)\omega^2(k^2-\tau)}
\text{ for } \abs{\lambda}< ck.
\end{multline}
Using $\alpha\omega^2 \theta^2 =1/16$ and $\beta\omega^2{\theta'}^2=1/16$ we can apply Lemma~\ref{help} and choose $c>0$ so small that the lefthand side in \eqref{sinterme} has the positive lower bound 
\begin{equation}
\min\left\{\frac{1+\tau}{4}, \frac{1-\tau}{4}\right\} 
\label{lower_bound_step}
\end{equation}
for all $k\in \N_{\odd}$ and all $\tau\in (-1,1)$. In order to find an upper bound for the righthand side of \eqref{sinterme} observe first that the map $\lambda \mapsto \frac{(\lambda+a)(\lambda+b)}{2\lambda+a+b}$ is strictly increasing in $\lambda>-\min\{a,b\}$ provided $a,b>0$. Hence using $\lambda<ck$ we obtain 
\begin{equation}
4 \frac{\sqrt{(\lambda+\alpha\omega^2(k^2-\tau))(\lambda+\beta\omega^2(k^2-\tau))}}{2\lambda+(\alpha+\beta)\omega^2(k^2-\tau)} < 4 \frac{\sqrt{\frac{ck}{k^2-\tau}+\frac{1}{16\theta^2}} \sqrt{\frac{ck}{k^2-\tau}+\frac{1}{16{\theta'}^2}}}{\frac{2ck}{k^2-\tau} + \frac{1}{16\theta^2} + \frac{1}{16{\theta'}^2}}.
\label{upper_bound_step}
\end{equation}
As we have seen in Lemma~\ref{nullinaufgehenderluecke2} we may use the inequality $\frac{k}{k^2-\tau} \leq \frac{2}{1-\tau}$ for all $\tau\in (-1,1)$, $k\in \N_{\odd}$, and hence the upper bound from \eqref{upper_bound_step} becomes 
\begin{equation}
4 \frac{\sqrt{(\lambda+\alpha\omega^2(k^2-\tau))(\lambda+\beta\omega^2(k^2-\tau))}}{2\lambda+(\alpha+\beta)\omega^2(k^2-\tau)} < 4 \frac{\sqrt{\frac{2c}{1-\tau}+\frac{1}{16\theta^2}} \sqrt{\frac{2c}{1-\tau}+\frac{1}{16{\theta'}^2}}}{\frac{1}{16\theta^2} + \frac{1}{16{\theta'}^2}}.
\label{upper_bound_step2}
\end{equation}
Combining \eqref{lower_bound_step} and \eqref{upper_bound_step2} we see that a sufficiently small value of $c$ (depending on $\tau, \alpha, \beta, \theta$) can be found provided 
\begin{equation}
\frac{16\theta(1-\theta)}{\theta^2+(1-\theta)^2} < \min\{1+\tau,1-\tau\}, \quad \mbox{ i.e., } \quad 0\leq |\tau| < 1- \frac{16\theta(1-\theta)}{\theta^2+(1-\theta)^2}.
\label{both_bounds_step}
\end{equation}
This requires $\theta$ or $\theta'$ to belong to $(0,\frac{1}{2}(1-\sqrt{7/9}))$. 
\end{proof} 

\subsection{Periodic potential in $H^r_{\per}(\R)$}  \label{PerPotBeispiel}

In our third example we consider the operators $L_k$, $k\in \Z_{\odd}$ given by the one-dimensional differential expression
\begin{align} \label{1dper_hr}
L_ku \coloneqq -u'' -\omega^2 (k^2-\tau) V(x) u \text{ on } \R
\end{align}
where $V\in H^r_{\per}(\R)$. Using the Fourier-coefficients $\hat V(n) := \frac{1}{2\pi}\int_0^{2\pi} V(x) e^{-inx}\,dx$ the space $H^r_{\per}(\R)$ is defined  as 
$$
H^r_{\per}(\R) := \{V\in L^2_{\loc}(\R): ( \hat V(n)(1+n^2)^{r/2})_{n\in\Z} \in l^2(\Z)\}
$$
with the norm $\|V\|_{H^r}  := \left(\sum_{n\in\Z} |\hat V(n)|^2 (1+n^2)^r\right)^{1/2}$. For $r \geq 1$ the operator $L_k$ is self-adjoint on the domain $D(L_k)=H^2(\R)$. 

The proof of the following lemma relies upon a recent result from \cite{Bruckner_Wayne}. There the authors consider the differential operator 
$L_V := -\frac{1}{V(x)}\frac{d^2}{dx^2}$ for $V\in H^r_{\per}(\R)$ with $V(x)\geq V_0>0$ for some $V_0\in \R$. The operator acting on the weighted Hilbert-space $L^2(\R, V\,dx)$ is self-adjoint with domain $H^2(\R)$. For $k\in \N$ let $\mu_k$ denote the $k$-th Dirichlet eigenvalue of $L_V$ and $\nu_k$ its $k$-th Neumann eigenvalue. Then $G_k(V) := \mu_k - \nu_k$ defines the signed gap-length. The band structure of the spectrum of the operator $L_V$ is encoded in the map 
$$
\mathcal{G}: V \mapsto \left( \frac{1}{\bigl(\int_0^{2\pi} \sqrt{V(x)}\,dx\bigr)^2}, (G_k(V))_{k\in \N}\right)
$$
and the main result of \cite{Bruckner_Wayne} says that $\mathcal{G}$ is a real-analytic isomorphism between a neighbourhood of $V=1$ in $H^r_{\per}(\R)$ and a neighbourhood of $(\frac{1}{4\pi^2}, (0)_{k\in \N})$ in the space $\R\times h^{r-2}$, where 
$$
h^{r-2}=\{(a_k)_{k\in \N}: (a_k (1+k^2)^{(r-2)/2})_{k\in\N} \in l^2(\N)\}.
$$

\begin{Lemma} \label{ex_of_potential}
Let $r \in [1,3/2)$ and $0<\gamma<\frac{3}{2}-r$. In every $H^r$-neighbourhood of $V_0\equiv 1$ there exists a potential $V\in H^r_{\per}(\R)$ which is positive and even with respect to $\pi$ and a value $\tau_0>0$ such that for $\omega = \pi/\int_0^{2\pi} \sqrt{V(x)}\,dx$ and all $\tau \in (-\tau_0,\tau_0)$, we have  $(-c\vert k\vert^\gamma,c\vert k\vert^\gamma)\subset \rho (L_k)$ for all $k\in \N_{\odd}$ for a suitable $c=c(V)>0$. 
\end{Lemma}

\begin{proof} 
Following the proof of Corollary~3 in \cite{Bruckner_Wayne} one can exploit the fact that the sequence $(k^\gamma)_{k\in \N}$ belongs to $h^{r-2}$ provided $0<\gamma< \frac{3}{2}-r$. Taking, e.g., $V := {\mathcal G}^{-1}(\omega^2/\pi^2, (a_k)_{k\in \N})$ with $(a_k)_{k\in \N}$ a small multiple $(k^\gamma)_{k\in \N}$ and $\omega\approx 1/2$ we obtain a function $V\in H^r_{\per}(\R)$ such that $(k^2\omega^2-d|k|^\gamma, k^2\omega^2+d|k|^\gamma) \in \rho(L_{V})$ for some $d>0$. Note that $\omega = \pi/\int_0^{2\pi} \sqrt{V(x)}\,dx$. Here we may assume that $V_{\min} = \min_\R V>0$. The fact that $(k^2\omega^2-d |k|^\gamma, k^2\omega^2+d|k|^\gamma)$ belongs to the resolvent set of $L_{V}$ is equivalent to 
$$
(-d |k|^\gamma, d|k|^\gamma) \subset \rho\Bigl(\frac{1}{V(x)}\bigl(-\frac{d^2}{dx^2}-\omega^2 k^2 V(x)\bigr)\Bigr) \mbox{ for all } k \in \N.
$$
If we set $\tilde c := d/2$ and  $\tau_0 := \frac{d-\tilde c}{\omega^2}=\frac{d}{2\omega^2}$ then $|\tau\omega^2+ t\tilde c|k|^\gamma| < d|k|^\gamma$ for $|\tau|<\tau_0$ and $|t|<1$ so that  
$$
(-\tilde c |k|^\gamma, \tilde c|k|^\gamma) \subset \rho\Bigl(\frac{1}{V(x)}\bigl(-\frac{d^2}{dx^2}-\omega^2 (k^2-\tau) V(x)\bigr)\Bigr)= \rho\Bigl(\frac{1}{V(x)} L_k\Bigr) \mbox{ for all } k \in \N \mbox{ and all } \tau \in (-\tau_0,\tau_0).
$$
Finally, using the monotonicity of band-edges with respect to $V(x)$ as stated in Lemma~\ref{monotonicity} we get that 
$$
(-V_{\min} \tilde c |k|^\gamma, V_{\min} \tilde c |k|^\gamma) \subset \rho(L_k) \mbox{ for all } k \in \N \mbox{ and all } \tau \in (-\tau_0,\tau_0)
$$
which finishes the proof.
\end{proof}

\section{Properties of $L_k$} \label{Sec:prop_Lk}

We assume that the potential $V$ satisfies one of the assumptions (V1), (V2) or (V3) from Theorem~\ref{Hauptresultat} and that $q(x)=\tau\omega^2 V(x)$ with $0\leq |\tau|<\tau_0$. Recall that $L_k: D(L_k)\subset L^2(\R)\to L^2(\R)$ given by 
$$
L_k := -\frac{d^2}{dx^2} - k^2\omega^2 V(x) + q(x).
$$
In this section we give two theorems on $k$-dependent estimates for bilinear forms associated to the operators $|L_k|$. The results are based on the spectral information for $L_k$ as stated in Lemmas~\ref{nullinaufgehenderluecke2}, \ref{nullinaufgehenderluecke_step} and \ref{ex_of_potential}. Recall in particular that there exist $k$-independent constants $c, \gamma>0$ such that 
\begin{equation} \label{spectral_info} 
(-c|k|^\gamma, c|k|^\gamma) \subset \rho(L_k) \,\mbox{ for all } \, k\in \Z_{\odd}
\end{equation} 
where $\gamma=1$ if $V$ satisfies (V1) or (V2) and $\gamma<\frac{3}{2}-r$ is a value associated with $V$ in case of (V3).  

\medskip

The operator $L_k$ is self-adjoint on a suitable domain $D(L_k)$ as explained in Sections~\ref{PerdeltaBeispiel}, \ref{PerstepBeispiel} and \ref{PerPotBeispiel}. In fact (cf. Theorem~VIII.15 in \cite{ReedSimon1}), $L_k$ is uniquely given by the semibounded, closed bilinear form
$$
b_{L_k}(u_k,v_k) = \int_\R u_k' \bar v_k' -\omega^2 (k^2-\tau) V(x) u_k \bar v_k\,dx, \quad u_k, v_k \in D(b_{L_k})=H^1(\R),
$$
where in case of assumption (V1) we use the notation $\int_\R \delta_{\per}(x) u_k \bar v_k \,dx := \sum_{n\in \Z} u_k(2\pi n)\bar v_k(2\pi n)$. If we denote by $(P^k_\lambda)_{\lambda\in\R}$ the projection-valued measure for $L_k$ then we find 
$$
\langle L_k u_k,v_k\rangle _{L^2(\R)} = b_{L_k}(u_k,v_k) = \int_{\R} \lambda d\langle P^k_\lambda u_k,v_k\rangle_{L^2(\R)} \text{ for } u_k\in D(L_k), v_k\in L^2(\R)
$$
and we can define the self-adjoint operator $|L_k|: D(|L_k|)=D(L_k)\subset L^2(\R) \to L^2(\R)$ and its corresponding bilinear form $b_{|L_k|}$ with domain $D(b_{|L_k|})=H^1(\R)$ by 
$$
\langle |L_k| u_k,v_k\rangle_{L^2(\R)} =b_{|L_k|}(u_k,v_k) = \int_{\R} |\lambda| d\langle P^k_\lambda u_k,v_k\rangle_{L^2(\R)} \text{ for } u_k\in D(|L_k|), v_k\in L^2(\R).
$$
Since $0\not\in \sigma(L_k)$ for all $k\in \Z_{\odd}$ we can introduce for $v\in L^2(\R)$ the splitting $v=v^++v^-$ with $v^\pm := P^{\pm,k} v$ and where
\begin{align*}
P^{+,k}v\coloneqq  \int_0^\infty 1 d\langle P^k_\lambda v,\cdot\rangle_{L^2(\R)},\ \ P^{-,k} v\coloneqq  \int_{-\infty}^0 1 d\langle P^k_\lambda v,\cdot\rangle_{L^2(\R)}.
\end{align*}
These splittings give rise to two new self-adjoint operators 
\begin{align} \label{ProjektionenkommutierenmitOperator}
L_k^\pm \colon P^{\pm,k} D(L_k)\subset P^{\pm,k} L^2(\R) \to P^{\pm,k} L^2(\R), \quad L_k^\pm u\coloneqq L_k u.
\end{align}
Their associated bilinear forms are restrictions of $b_{L_k}$ to $D(b_{L_k^\pm})\times D(b_{L_k^\pm})$ with $D(b_{L_k^\pm})= P^{\pm,k} D(b_{L_k})= P^{\pm,k} H^1(\R)$. Note that $L_k= L_k^++L_k^-$ and $|L_k|=L_k^+-L_k^-$. 


\medskip

\begin{Theorem} \label{Abschnachschlag}
There is $c>0$ such that
\begin{align} \label{Abschfuerurechts}
b_{|L_k|}(v,v) \geq c\abs{k}^\gamma \Vert v\Vert_{L^2(\R)}^2 \text{ for all } v\in H^1(\R) \text{ and all } k\in\Z_{\odd}
\end{align}
with $\gamma=1$ if $V$ satisfies (V1) or (V2) and $\gamma<\frac{3}{2}-r$ is a value associated with $V$ in case of (V3).
\end{Theorem}

\begin{proof}
Recall that for a self-adjoint operator $A\colon D(A)\subset L^2(\R) \to L^2(\R)$ which is bounded from below, we have
\begin{align} \label{Raylieghfuerposdef}
\inf_{f\in D(A)} \frac{\ang{Af,f}_{L^2(\R)}}{\Vert f\Vert_{L^2(\R)}^2} = \inf \sigma (A).
\end{align}
The idea is now to use the splitting of the indefinite operator $L_k$ into a positive definite and a negative definite operator $L_k^\pm$, apply \eqref{Raylieghfuerposdef} and then use the density of $D(L_k)$ in $H^1(\R)$. From \eqref{Raylieghfuerposdef} and \eqref{spectral_info} we conclude that
\begin{align} \label{aufDLplus}
\inf_{v\in P^{+,k} D(L_k)} \frac{\langle L_k^+ v,v\rangle_{L^2(\R)}}{\Vert v\Vert^2_{L^2(\R)}} \geq c \vert k\vert^\gamma,\quad \inf_{v\in P^{-,k} D(L_k)} -\frac{\langle L_k^- v,v\rangle_{L^2(\R)}}{\Vert v\Vert^2_{L^2(\R)}} \geq c \vert k\vert^\gamma
\end{align}
for some $c>0$. By \eqref{aufDLplus} one obtains for all $v\in D(|L_k|)=D(L_k)$
\begin{align*}
b_{|L_k|}(v,v)= \langle |L_k| v,v\rangle_{L^2(\R)} &=
\langle L_k^+ P^{+,k} v,P^{+,k} v\rangle _{L^2(\R)}-\langle L_k^- P^{-,k}v,P^{-,k}v\rangle  \\ 
&\geq c\vert k\vert^\gamma \paren{\Vert P^{+,k}v\Vert^2_{L^2(\R)}+\Vert P^{-,k} v\Vert^2_{L^2(\R)}} = c \vert k\vert^\gamma \Vert v\Vert^2_{L^2(\R)}
\end{align*}
and \eqref{Abschfuerurechts} then follows from the density statement mentioned above. 
\end{proof}

The benefit of an estimate like \eqref{Abschfuerurechts} lies in the $k$-dependence. In the following result we construct a similar lower bound with $\Vert v'\Vert_{L^2(\R)}^2$ instead of $\Vert v\Vert_{L^2(\R)}^2$ in the right hand side of \eqref{Abschfuerurechts}.

\begin{Theorem} \label{ZielAbschnittNonlinearity}
There is a constant $\tilde c>0$ such that
\begin{align} \label{Abschfuernablaurechts}
b_{|L_k|}(v,v) \geq \tilde c |k|^{\delta} \Vert v'\Vert_{L^2(\R)}^2 \text{ for all } v\in H^1(\R) \text{ and all } k\in\Z_{\odd}
\end{align}
with $\delta=\gamma-4=-3$ if $V$ satisfies (V1), $\delta=\gamma-2=-1$ if $V$ satisfies (V2) and $\delta=\gamma-2<-\frac{1}{2}-r$ is a value associated with $V$ in case of (V3).
\end{Theorem}

\begin{proof} For $k\in \Z_{\odd}$ let $V_k(x)=-\omega^2k^2 V(x)+q(x)= -\omega^2(k^2-\tau)V(x)$. We prove \eqref{Abschfuernablaurechts} by several case distinctions depending on the assumption on $V$. Let $\lambda\in (0,1)$ be fixed for the whole proof. We begin with $V$ satisfying (V1).

\smallskip

\noindent
Case 1: Let $v\in D(b_{L_k^+})$. We distinguish two cases. As usual we use the notation $\int_\R \delta_{\per}(x) |v|^2\,dx := \sum_{n\in \Z} |v(2\pi n)|^2$. 

a): $\int_{\R} |v'|^2 +\frac{1}{1-\lambda} V_k |v|^2\,dx \geq 0$: Then we directly obtain $\int_{\R} |v'|^2 + V_k |v|^2\,dx \geq \lambda\int_{\R} |v'|^2 dx$.

b): $-\left(\int_{\R} |v'|^2+\frac{1}{1-\lambda} V_k |v|^2\,dx\right) \geq 0$: Using Lemma~\ref{fundablpktwgegennull} we get that for every $\varepsilon>0$ 
\begin{equation} \label{PunktabschSummemitepsilon}
\int_\R -V_k |v|^2\,dx \leq Ck^2\left(\varepsilon \|v'\|_{L^2(\R)}^2 + (1+\frac{1}{\varepsilon}) \|v\|_{L^2(\R)}^2\right).
\end{equation}
Therefore,
\begin{equation} \label{konsequenz1}
\int_{\R} |v'|^2\, dx \leq -\frac{1}{1-\lambda} \int_\R  V_k|v|^2\,dx  \leq \frac{Ck^2}{1-\lambda} \left(\varepsilon \|v'\|_{L^2(\R)}^2 + (1+\frac{1}{\varepsilon}) \|v\|_{L^2(\R)}^2\right).
\end{equation}
In particular, for $\varepsilon= \varepsilon_k\coloneqq \frac{1-\lambda}{2 Ck^2}$ we have 
\begin{equation} \label{konsequenz2}
\Vert v'\Vert_{L^2(\R)}^2 \leq \frac{2C k^2}{1-\lambda} (1+\frac{1}{\varepsilon})\Vert v\Vert^2_{L^2(\R)}\leq \bar C |k|^4\|v\|_{L^2(\R)}^2.
\end{equation}
Together with Theorem~\ref{Abschnachschlag} we conclude
\begin{align} \label{Falleinsb}
\frac{\int_{\R} |v'|^2+V_k |v|^2\,dx}{\int_{\R} |v'|^2 dx} = \frac{\paren{\int_{\R}|v'|^2 +V_k|v|^2\,dx}}{\Vert v\Vert^2_{L^2(\R)}} \frac{\Vert v\Vert_{L^2(\R)}^2}{\Vert v'\Vert^2_{L^2(\R)}} \geq \tilde C\vert k\vert^{-3}.
\end{align}
Merging Case 1a) and \eqref{Falleinsb} we deduce $b_{|L_k|}(v,v) \geq \tilde c\vert k\vert^{-3} \|v'\|_{L^2(\R)}^2$ for all $v\in D(b_{L_k^+})$ and some constant $\tilde c>0$.

\smallskip

\noindent
Case 2: Let $v\in D(b_{L_k^-})$, i.e., $\int_{\R} |v'|^2 + V_k|v|^2\,dx \leq -c\vert k\vert \int_{\R} |v|^2 dx$. By \eqref{PunktabschSummemitepsilon} with $\varepsilon=\varepsilon_k=\frac{1}{2Ck^2}$ we deduce
\begin{align} \label{wiederOvonkhochdrei}
\int_{\R} |v'|^2 dx \leq 2\left( C|k|^2(1+\frac{1}{\varepsilon_k})-c|k|\right) \Vert v\Vert^2_{L^2(\R)} \leq \bar C|k|^4 \Vert v\Vert^2_{L^2(\R)}.
\end{align}
In analogy to the first case we now conclude
\begin{align*}
\frac{-\left(\int_{\R} |v'|^2 +V_k |v|^2\,dx\right)}{\Vert v'\Vert^2_{L^2(\R)}} = \frac{-\left(\int_{\R} |v'|^2 + V_k|v|^2\,dx\right)}{\Vert v\Vert^2_{L^2(\R)}} \frac{\Vert v\Vert^2_{L^2(\R)}}{\Vert v'\Vert^2_{L^2(\R)}} \geq \tilde c\vert k\vert \frac{\Vert v\Vert^2_{L^2(\R)}}{\Vert v'\Vert^2_{L^2(\R)}}
\end{align*}
and due to \eqref{wiederOvonkhochdrei} the fraction $\frac{\Vert v\Vert^2_{L^2(\R)}}{\Vert v'\Vert^2_{L^2(\R)}}$ is of order $\vert k\vert^{-4}$ which establishes our claim in the case $v\in D(b_{L_k^-})$.

\smallskip

Finally, merging the two estimates for $D(b_{L_k^+})$ and $D(b_{L_k^-})$ from Case 1 and Case 2 we end up with
\begin{align*}
b_k(v^+,v^+)-b_k(v^-,v^-) \geq \tilde c\vert k\vert^{-3}\int_{\R} \paren{\paren{|v^{{+'}}|}^2+\paren{|v^{{-'}}|}^2} dx \geq \frac{\tilde{c}}{2}\vert k\vert^{-3}\int_{\R} |v'|^2 dx
\end{align*}
for a constant $\tilde{c}>0$ and the proof is done. 

\smallskip

Let us now discuss the situation where $V$ satisfies (V2) or (V3). The proof follows the same patterns as before. Let us indicate the changes. Note that now $V\in L^\infty(\R)$. Case 1a) is unchanged. In Case 1b) inequality \eqref{PunktabschSummemitepsilon} is replaced by 
\begin{equation}
\int_\R -V_k |v|^2\,dx \leq C k^2 \|v\|_{L^2(\R)}^2.
\end{equation}
Therefore, using the analogy of the steps \eqref{konsequenz1}, \eqref{konsequenz2} we arrive instead of \eqref{Falleinsb} at 
\begin{align} 
\frac{\int_{\R} |v'|^2+V_k |v|^2\,dx}{\int_{\R} |v'|^2 dx} \geq \tilde C|k|^{-1}.
\end{align}
In Case 2) inequality \eqref{wiederOvonkhochdrei} is replaced by
\begin{align} \label{wiederOvonkhochdrei_r=1/2}
\Vert v'\Vert^2_{L^2(\R)} \leq Ck^2 \Vert v\Vert^2_{L^2(\R)}
\end{align}
which leads to 
$$
\frac{-\left(\int_{\R} |v'|^2+V_k|v|^2\,dx\right)}{\Vert v'\Vert^2_{L^2(\R)}}  \geq \tilde c |k|^{-1}.
$$
The proof is then	 finished as before.
\end{proof}

\section{The functional analytic framework for breathers} \label{Sec:fa}

In this section we define a suitable Hilbert space in which we seek for solutions. We use the projection-valued measure $(P_\lambda^k)_{\lambda\in \R}$ defined in the previous section to represent $L_k = \int_\R \lambda\, dP_\lambda^k$.  
\begin{Definition} \label{def_H} 
Define the Hilbert space $\mc{H}$ over the field $\R$ by
\begin{equation}
\mc{H} \coloneqq \Bigl\{ \tilde{u}=(u_k)_{k\in \Z_{\odd}}: u_k\in H^1(\R), \bar u_k =u_{-k} \text{ for all } k\in\Z_{\odd} \text{  and } \sum_{k\in\Z_{\odd}} \int_\R  |\lambda| \langle P_\lambda^k u_k,u_k\rangle_{L^2(\R)}\,d\lambda <\infty \Bigr\}
\end{equation}
with the canonical inner product and norm
\begin{align*}
\ang{\tilde{u},\tilde{v}}_{\mc{H}} \coloneqq \sum_{k\in\Z_{\odd}} \int_\R \vert \lambda \vert  \langle P_\lambda^k u_k,u_k\rangle\,d\lambda\ \ \  \text{ and } \ \ \ \Vert \tilde{u}\Vert_{\mc{H}} \coloneqq \sqrt{\ang{\tilde{u},\tilde{u}}} \text{ for } \tilde{u}, \tilde{v}\in\mc{H}.
\end{align*}
\end{Definition}

Next, we introduce projections $\mc{P}^+$ and $\mc{P}^-$ to deal with the indefinite character of the problem. Let
\begin{align*}
\mc{H}^+ &\coloneqq \mc{P}^+\mc{H}\coloneqq \{\tilde{u}\in \mc{H}: P^{-,k} u_k= 0 \text{ for all } k\in \Z_{\odd} \},\\
\mc{H}^- &\coloneqq \mc{P}^-\mc{H}\coloneqq \{\tilde{u}\in \mc{H}: P^{+,k} u_k= 0 \text{ for all } k\in \Z_{\odd} \}
\end{align*}
and set $\tilde{u}^\pm\coloneqq \mc{P}^\pm\tilde{u}$. The potentials $V$ are constructed such tha for all $k\in \Z_{\odd}$ we have $0\not\in \sigma(L_k)$ so that $u_k=0\Leftrightarrow P^{+,k}u_k=P^{-,k}u_k=0$. Therefore we obtain the splitting $\mc{H}=\mc{H}^+ \oplus \mc{H}^-$. If we consider the bilinear form $B\colon \mc{H}\times\mc{H}\to \C$ defined by
\begin{align*}
B(\tilde{u},\tilde{v})=\sum_{k\in\Z_{\odd}} b_{L_k}(u_k,v_k) \text{ for } \tilde{u},\tilde{v}\in\mc{H}
\end{align*}
then we obtain 
\begin{align} \label{SplittingvonHRmcH}
B(\tilde{u},\tilde{u})=\Vert \tilde{u}^+\Vert^2_{\mc{H}} - \Vert \tilde{u}^-\Vert^2_{\mc{H}} \text{ for all } \tilde{u}\in\mc{H}.
\end{align}
Hence, $\Vert\tilde{u}\Vert^2_{\mc{H}} = \Vert\tilde{u}^+\Vert^2_{\mc{H}}+\Vert\tilde{u}^-\Vert^2_{\mc{H}}$, and in particular $\Vert\tilde{u}^+\Vert_{\mc{H}}, \Vert \tilde{u}^-\Vert_{\mc{H}}\leq \Vert\tilde{u}\Vert_{\mc{H}}$ for all $\tilde{u}\in\mc{H}$.

\medskip

Now we establish integrability of the composite function $u(x,t)=\sum_{k\in\Z_{\odd}} u_k(x)e^{ik\omega t}$ in space and time as expressed by the following theorem. The proof, which is rather complex, is given in Section~\ref{sec:proof_for_S}.

\begin{Theorem} \label{HauptresultatSection5.5}
With $D=\R\times (0,T)$ the linear operator $\mc{S}\colon \mc{H}\to L^q(D)$ given by 
\begin{align*}
(\mc{S}\tilde{u})(x,t)\coloneqq \sum_{k\in\Z_{\odd}} u_k(x) e^{ik\omega t}
\end{align*}
is one-to-one and bounded for all $q\in [2,q^\ast)$ where 
$$
q^\ast = \left\{\begin{array}{ll}
3 & \mbox{ in case (V1)}, \vspace{\jot}\\
4 & \mbox{ in case (V2)}, \vspace{\jot}\\
\frac{4}{2-\gamma} & \mbox{ in case (V3)}.
\end{array}
\right.
$$
For the same values of $q$ the operator $\mc{S}\colon \mc{H}\to L^q(K)$ is compact for every compact set $K\subset \overline{D}$.  
\end{Theorem}

\begin{Remark} \label{crit_sob} In case of assumptions (V2) and (V3) the above embedding $\mc{S}\colon \mc{H}\to L^q(D)$ is bounded also for $q=q^\ast$. This is due to the fact that in this case we show in the proof of Theorem~\ref{HauptresultatSection5.5} the embedding $\mc{S}: \mc{H}\to H^\frac{\gamma}{2}(D)$, cf. \eqref{vorauss_fuer_einbettung} with $\rho=2$. In this case however, it is known that the embedding  $H^\frac{\gamma}{2}(D)\to L^q(D)$ not only holds for $2\leq q< q^\ast= \frac{4}{2-\gamma}$ but also for the endpoint $q=q^\ast=\frac{4}{2-\gamma}$, cf. \cite{Hitchhiker}. In the case of (V1) this question of the existence of the endpoint embedding is unknown to us because in this case the underlying fractional Sobolev space is anisotropic with respect to the directions $x$ and $t$ and hence the usual proof of the endpoint embedding via the Hardy-Littlewood-Sobolev inequality does not work.
\end{Remark}

\section{Minimization on the generalized Nehari manifold} \label{MinimgenNM}

Now we find the time-periodic solution of \eqref{Einl2scalar}$_\pm$ as a minimizer of a functional $J$ on the so-called generalized Nehari manifold. We are using Theorem~35, Chapter~4 from \cite{SW}, where an abstract result is given that guarantees the existence of minimizer of an indefinite functional on the generalized Nehari manifold. We first treat the ''$+$''-case in \eqref{Einl2scalar}$_\pm$. At the end of this section we explain how the ''$-$''-case can be treated. Let $J\colon \mc{H} \to \R$ be given by
\begin{align*}
J(\tilde{u}) \coloneqq J_0(\tilde{u})-J_1(\tilde{u}) 
\end{align*}
with
\begin{align*}
J_0(\tilde{u})\coloneqq \frac{1}{2}B(\tilde u,\tilde u), \quad J_1(\tilde{u})\coloneqq  \frac{1}{T}\int_D F(x,\mc{S}\tilde{u})\, d(x,t)
\end{align*}
and where $\mc{S}$ is the operator from Theorem~\ref{HauptresultatSection5.5} which reproduces $u(x,t)$ from the Fourier-variables $\tilde u = (u_k)_{k\in \Z_{\odd}}\in \mc{H}$. Due to assumption (H1) and Theorem~\ref{HauptresultatSection5.5} the functional $J$ is well-defined on $\mc{H}$. The generalized Nehari manifold is defined as 
\begin{align*}
\mc{M}\coloneqq \{\tilde{u}\in \mc{H}\setminus \mc{H}^-: J'(\tilde{u})[\tilde{u}]=0 \text{ and } J'(\tilde{u})[\tilde{v}]=0 \text{ for all } \tilde{v}\in \mc{H}^-\}.
\end{align*}
Moreover, for $\tilde{u}\in\mc{H}$ we set
\begin{align*}
\mc{H}(\tilde{u})\coloneqq \R^+\tilde{u} \oplus \mc{H}^- = \R^+ \tilde{u}^+ \oplus \mc{H}^-,
\end{align*}
where $\R^+= [0,\infty)$. Finally, let $S$ denote the unit ball in $\mc{H}$ and define $S^+\coloneqq S\cap \mc{H}^+$.

\medskip

By standard calculations (compare Proposition~1.12 in \cite{Willem}) we deduce $J\in C^1(\mc{H})$ and 
\begin{align*}
J'(\tilde{u})[\tilde{v}] = J_0'(\tilde u)[\tilde v]-J_1'(\tilde u)[\tilde v]= B(\tilde u, \tilde v)-\frac{1}{T}\int_D \vert \mc{S}\tilde{u}\vert^{p-1} \mc{S}\tilde{u} \overline{\mc{S}\tilde{v}} d(x,t).
\end{align*}
Notice that $\tilde u, \tilde v\in \mc{H}$ imply that $\mc{S}\tilde u, \mc{S}\tilde v$ are read-valued functions and that $J_0'(\tilde u)[\tilde v], J_1'(\tilde u)[\tilde v]\in \R$. The verification of $J'[\tilde u]=0$ for a suitable $\tilde u\in \mc{H}$ is a key point in this section. We simplify this task by the following lemma. The proof is given in the Appendix.

\begin{Lemma} For $k\in \Z_{\odd}$ let 
$$
\mc{H}_{k,\mono} \coloneqq \left\{\tilde{\phi}=(\phi_l)_{l\in\Z_{\odd}}: \phi_l = \phi\delta_{kl} \mbox{ for some } \phi \in C_c^\infty(\R)\right\}. 
$$ 
Let $\tilde u\in\mc{H}$. Then the following are equivalent:
\begin{itemize}
\item[(i)] for all $k\in \Z_{\odd}$ we have $J'(\tilde u)[\tilde \phi]=0$ for all $\tilde\phi\in \mc{H}_{k,\mono}$
\item[(ii)] $J'(\tilde u)=0$.
\end{itemize}
\label{ueberpruefung_cp}
\end{Lemma}

\begin{Remark} The set $\mc{H}_{k,\mono}$ consists of Fourier-modes where only the frequency $k\omega$ is occupied while all other frequencies $l\omega$ with $l\not=k$ are not occupied. Because of the missing conjugation-symmetry $\mc{H}_{k,\mono}$ is not a subset of $\mc{H}$. Nevertheless, the functionals $J$, $J'$ as well as the map $\mc{S}$ naturally extend as continuous maps to $\overline{\mc{H}}_{k,\mono}$. 
\end{Remark}

We start verifying the assumption $(B_1)$, (i) and (ii) of Theorem~35 in \cite{SW}.

\begin{Lemma} \label{FormelNichtlin1}
The following statements hold true:
\begin{enumerate}
\item $J_1$ is weakly lower semicontinuous, 
\begin{align} \label{BeinsersteVor}
J_1(0)=0 \quad \text{ and } \quad \frac{1}{2}J_1'(\tilde{u})[\tilde{u}]> J_1(\tilde{u})>0 \text{ for } \tilde{u}\neq 0.
\end{align}
\item $\lim_{\tilde{u}\to 0} \frac{J_1'(\tilde{u})}{\Vert\tilde{u}\Vert_{\mc{H}}}=0$ and $\lim_{\tilde{u}\to 0} \frac{J_1(\tilde{u})}{\Vert \tilde{u}\Vert^2_{\mc{H}}}=0$.
\item For a weakly compact set $U\subset \mc{H}\setminus\{0\}$ we have $\lim_{s\to\infty} \frac{J_1(s\tilde{u})}{s^2} = \infty$ uniformly w.r.t. $\tilde{u}\in U$.
\end{enumerate}
\end{Lemma}

\begin{proof}
(a) Note that (H2) and (H3) imply $f(x,s)s > 2F(x,s)>0$ for all $s\not =0$. Since $\mc{S}\colon\mc{H}\to L^{p+1}(D)$ is one-to-one this implies \eqref{BeinsersteVor}. The weak lower-semicontinuity of $J_1$ follows from Fatou's lemma and the fact that a weakly convergent sequence $(\tilde u_n)_{n\in\N}$ in $\mc{H}$ has the property that $(\mc{S}\tilde u_n)_{n\in\N}$ converges weakly in $L^2(D)$, strongly in $L^2(K)$ for every compact subset $K\subset \overline{D}$ and (for a subsequence) pointwise almost everywhere in $D$.

\smallskip

(b) It follows from (H2) that for every $\epsilon>0$ there is $C_\epsilon>0$ such that $|f(x,s)|\leq \epsilon |s|+ C_\epsilon |s|^p$ and hence $0\leq F(x,s) \leq \frac{\epsilon}{2}s^2 + \frac{C_\epsilon}{p+1}|s|^{p+1}$. The claim is then immediate by the embedding provided by Theorem~\ref{HauptresultatSection5.5}.

\smallskip

(c) Let $U\subset \mc{H}\setminus\{0\}$ be weakly compact. To prove the claim it is sufficient to show that for every sequence $(\tilde u_n)_{n\in \N}$ in $U$ and every sequence $s_n\to \infty$ we have $\liminf_{n\in \N} \frac{J_1(s_n\tilde{u}_n)}{s_n^2} = \infty$. Up to a subsequence we have with $u_n:= \mc{S}\tilde u_n$ that $u_n\to u$ a.e. in $D$ as $n\to \infty$ and $u\not=0$ on a set $A\subset D$ of positive measure. By (H4)  
$$
\lim_{n\to \infty} \frac{F(x, s_n u_n(x,t))}{s_n^2 u_n(x,t)^2} = \infty \mbox{ a.e. on } A 
$$
so that by Fatou's Lemma 
$$
\liminf_{n \in N} \frac{J_1(s_n \tilde{u}_n)}{s_n^2} \geq  \liminf_{n\in N} \int_A \frac{F(x,s_nu_n(x,t))}{s_n^2 u_n(x,t)^2} u_n(x,t)^2\,d(x,t) = \infty.
$$
\end{proof}

Assumption $(B_2)$ of Theorem~35 in \cite{SW} is guaranteed by the next result.

\begin{Lemma} \label{FormelNichtlin2}
The following statements hold true:
\begin{enumerate}
\item For each $\tilde{w}\in\mc{H}\setminus \mc{H}^-$ there exists a unique nontrivial critical point $m_1(\tilde{w})$ of $J|_{\mc{H}(\tilde{w})}$. Moreover, $m_1(\tilde{w})\in \mc{M}$ is the unique global maximizer of $J|_{\mc{H}(\tilde{w})}$ as well as $J(m_1(\tilde{w}))>0$.
\item There exists $\delta>0$ such that $\Vert m_1(\tilde{w})^+\Vert_{\mc{H}} \geq \delta$ for all $\tilde{w}\in \mc{H}\setminus \mc{H}^-$.
\end{enumerate}
\end{Lemma}

\begin{proof}
(a) We can directly follow the lines of proof of Proposition~39 in \cite{SW}.

\smallskip

(b) First, consider $\tilde{v}\in\mc{H}^+$. Then we have 
$
\lim_{\tilde{v}\to 0} \frac{J(\tilde{v})}{\Vert\tilde{v}\Vert^2_{\mc{H}}} =\frac{1}{2}
$
due to Lemma~\ref{FormelNichtlin1}~(b). Thus there is $\rho_0>0$ s.t. $J(\tilde{v})\geq \frac{1}{4}\Vert\tilde{v}\Vert^2_{\mc{H}}$ for all $\tilde{v}\in\mc{H}^+$ with $\Vert\tilde{v}\Vert_{\mc{H}} \leq \rho_0$. Hence for $\rho \in (0,\rho_0)$ we find $\eta=\frac{\rho^2}{4}$ with $J(\tilde{v})\geq \eta$ for all $\tilde{v}\in\mc{H}^+$ with $\Vert\tilde{v}\Vert_{\mc{H}}=\rho$. Now, let $\tilde{w}\in \mc{H}\setminus\mc{H}^-$. Due to the structure of $J$ we infer that 
\begin{align} \label{tildewteileins}
\frac{\Vert m_1(\tilde{w})^+\Vert^2_{\mc{H}}}{2} \geq J(m_1(\tilde{w})).
\end{align} 
Since $m_1(\tilde{w})$ is the maximizer of $J|_{\mc{H}(\tilde{w})}$ we conclude
\begin{align} \label{tildewteilzwei}
J(m_1(\tilde{w}))\geq J\paren{\rho \frac{\tilde{w}^+}{\Vert \tilde{w}^+\Vert_{\mc{H}}}} \geq \eta.
\end{align}
and the combination of \eqref{tildewteileins} and \eqref{tildewteilzwei} finishes the proof of part (b).
\end{proof}

\begin{Lemma} \label{NichtLin5}
Any Palais-Smale sequence $\paren{\tilde{u}_n}_{n\in\N}$ of $J|_{\mc{M}}$ is bounded.
\end{Lemma}

\begin{proof} The following proof is similar to the proof of Theorem~40 in \cite{SW}. 

\smallskip

\noindent
{\em Step 1:} Suppose for contradiction that $(\tilde u_n)_{n\in \N}$ is an unbounded Palais-Smale sequence for $J$. By selecting a subsequence we may assume that $\|\tilde u_n\|_{\mc{H}}\to \infty$ and that $\tilde v_n := \tilde u_n/\|\tilde u_n\|_{\mc{H}}$ has the property that $\tilde v_n\rightharpoonup \tilde v$ as $n\to \infty$. Note that 
\begin{equation} \label{rescale}
0 \leq \frac{J(\tilde u_n)}{\|\tilde u_n\|_{\mc{H}}^2} = \frac{1}{2} \|\tilde v_n^+\|_{\mc{H}}^2 - \frac{1}{2}\|\tilde v_n^-\|_{\mc{H}}^2 - \frac{J_1(\|\tilde u_n\|_{\mc{H}} \tilde v_n)}{\|\tilde u_n\|_{\mc{H}}^2}.
\end{equation}
If $\tilde v\not =0$ then we can apply Lemma~\ref{FormelNichtlin1}(c) to the weakly compact set $U=\{\tilde v_n: n\in \N\}\cup{\tilde v}$ which does not contain $0$ and find that the expression $\frac{J_1(\|\tilde u_n\|_{\mc{H}} \tilde v_n)}{\|\tilde u_n\|_{\mc{H}}^2} \to \infty$ as $n\to \infty$. This is not compatible with \eqref{rescale} and hence the weak limit $\tilde v=0$. 

\smallskip

\noindent
{\em Step 2:} Next, let us show that $\mc{S}\tilde v_n^+ \to 0$ in $L^{p+1}(D)$ is impossible.
Since $J_1\geq 0$ we conclude from \eqref{rescale} that $\|\tilde v_n^-\|_{\mc{H}}^2 \leq \|\tilde v_n^+\|_{\mc{H}}^2$ which together with $\|\tilde v_n^-\|_{\mc{H}}^2 + \|\tilde v_n^+\|_{\mc{H}}^2=1$ implies that $\|\tilde v_n^+\|_{\mc{H}}^2\geq 1/2$.  Next, note by (H1) and (H2) that for every $\epsilon>0$ there exists $C_\epsilon>0$ such that $|F(x,s)| \leq \epsilon s^2 +C_\epsilon |s|^{p+1}$ so that for every $\tilde u\in \mc{H}$ one has $0\leq J_1(\tilde u) \leq \epsilon C \|\tilde u\|_{\mc{H}}^2 + \bar C_\epsilon \|\mc{S}\tilde u\|_{L^{p+1}(D)}^{p+1}$ by Theorem~\ref{HauptresultatSection5.5}. Since $\tilde v_n$ is a positive multiple of $\tilde u_n$ (which itself belongs to $\mc{M}$) Lemma~\ref{FormelNichtlin2}(a) together with the preceeding inequality for $J_1$ and $\|\tilde v_n^+\|_{\mc{H}}^2\geq 1/2$ imply that for any $s>0$
\begin{equation} \label{unmoeglich}
J(\tilde u_n) \geq J(s \tilde v_n^+) = \frac{s^2}{2} \|\tilde v_n^+\|^2_{\mc{H}}- J_1(s \tilde v_n^+) \geq \frac{s^2}{4} -\epsilon C s^2\|\tilde v_n^+\|_{\mc{H}}^2 - \bar C_\epsilon |s|^{p+1}\|\mc{S}\tilde v_n^+\|_{L^{p+1}(D)}^{p+1}.
\end{equation}
The left hand side is bounded sind $(\tilde u_n)_{n\in \N}$ is a Palais-Smale sequence, and $\|v_n^+\|_{\mc H}$ is also bounded by weak convergence. Thus, choosing $\epsilon>0$ small enough but $s>0$ large, we cannot have $\|\mc{S}\tilde v_n^+\|_{L^{p+1}(D)}^{p+1}\to 0$ as $n\to \infty$ in \eqref{unmoeglich}. 

\smallskip
\noindent
{\em Step 3:} Shifting $\tilde v_n^+$. By Step 2, i.e., $\mc{S} \tilde v_n^+$ not converging to $0$ in $L^{p+1}(D)$, Lemma~\ref{ConccompLemma} applies and we find $\delta>0$, a sequence $(y_n)_{n\in\N}$ in $D$ and a subsequence of $(\tilde{v}_n)_{n\in\N}$ (again denoted by $(\tilde{v}_n)_{n\in\N}$) such that
\begin{align} \label{Beinsyn}
\int_{B_1(y_n)} \vert\mc{S}\tilde{v}_n^+\vert^2 d(x,t) \geq \delta>0 \text{ for all } n\in\N.
\end{align}
Next we shift $\tilde{v}_n^+$ in such a way that we can make use of compact embeddings for the shifted sequence. For the centers $y_n=(x_n,t_n)^T$ of the balls appearing in \eqref{Beinsyn} we have $x_n=2\pi m_n+r_n$ for some $m_n\in\Z, r_n\in [0,2\pi)$. The shifted centers are denoted by $y_n'\coloneqq (r_n,t_n)^T \in [0,2\pi)\times [0,T)$. Let us define new functions $\tilde v_n^\ast$ by 
\begin{align*}
\tilde{v}_n^{\ast}(\cdot) \coloneqq \tilde{v}_n(\cdot +2\pi m_n).
\end{align*}
Note that shifting does not change norms in $\mc{H}$ and shifting commutes with the spectral projections $\mc{P}^{\pm}$ since the operators $L_k$ are shift invariant, i.e.,$\tilde v_n^{\ast,+}=\tilde v_n^{+,\ast}$. If we set $\tilde{B}\coloneqq [-1,2\pi+1]\times [-1,T+1]$ then $B_1(y_n')\subset \tilde B$ for all $n\in\N$. Moreover, \eqref{Beinsyn} entails
\begin{align*}
\int_{\tilde{B}} \vert\mc{S}\tilde{v}_n^{\ast,+}\vert^2 d(x,t) \geq \int_{B_1(y_n')} \vert\mc{S}\tilde{v}_n^{\ast,+}\vert^2 d(x,t) = \int_{B_1(y_n)} \vert\mc{S}\tilde{v}_n^+\vert^2 d(x,t)\geq \delta \text{ for all } n\in\N.
\end{align*}
We know that (up to a subsequence) $\tilde v_n^{\ast}\rightharpoonup \tilde v^{\ast}\in \mc{H}$ as $n\to \infty$. The compact embedding into $L^2(\tilde B)$ from Theorem~\ref{HauptresultatSection5.5} yields $\Vert \mc{S}\tilde v^{\ast,+}\Vert_{L^2(D)}\neq 0$, i.e., $\tilde v^{\ast,+}\not =0$ and hence $\tilde v^{\ast} \not=0$. This, however, contradicts the observation $\tilde v^{\ast}= w\mbox{-}\lim_{n\to \infty} \tilde v_n^\ast=0$ from the beginning of the proof. This contradiction finishes the proof of the boundedness of Palais-Smale sequences of $J|_{\mc{M}}$. 
\end{proof}

Finally, we can turn to our overall goal of this section and verify the following statement.

\begin{Theorem} \label{MinThm}
The functional $J$ admits a ground state, i.e., there exists $\tilde{u}\in \mc{M}$ such that $J'(\tilde u)=0$ and \mbox{$J(\tilde{u})=\inf_{\tilde{v}\in\mc{M}} J(\tilde{v})$.} 
\end{Theorem}

The proof requires the following variant of a concentration-compactness Lemma of P.~L.~Lions, cf. Lemma~1.21 in \cite{Willem} for a similar result in non-fractional Sobolev-spaces. Its proof is given in the Appendix. Recall that we interpret $\tilde{u}\in \mc{H}$ as a function on $D$ which is continued to $\R^2$ periodically w.r.t. the second component. This is needed since in the following lemma the balls $B_r(y)$ can exceed the set $D$.

\begin{Lemma} \label{ConccompLemma}
Let $q\in [2,p^\ast+1)$ and $r>0$ be given with $p^\ast$ from Theorem~\ref{Hauptresultat}. Moreover, let $(\tilde{u}_n)_{n\in\N}$ be a bounded sequence in $\mc{H}$ and 
\begin{align} \label{Bedconccomp}
\sup_{z\in D} \int_{B_r(z)} \vert \mc{S}\tilde{u}_n\vert^q d(x,t) \to 0 \text{ as } n\to\infty.
\end{align}  
Then $\mc{S}\tilde{u}_n\to 0$ in $L^{\tilde{q}}(D)$ as $n\to\infty$ for all $\tilde{q}\in (2,p^\ast+1)$.
\end{Lemma}

\noindent
\textit{Proof of Theorem~\ref{MinThm}:} Conditions (B1), (B2) and (i) and (ii) of Theorem~35 in \cite{SW} are fulfilled, and only (iii) does not hold so that $J$ does not satisfy the Palais-Smale condition. As a consequence, Theorem~35 in \cite{SW} only provides a minimizing Palais-Smale $(\tilde{u}_n)_{n\in \N}$ in $\mc{M}$ with 
$J'(\tilde{u}_n)\to 0$ as $n\to\infty$. Lemma~\ref{NichtLin5} guarantees that $(\tilde{u}_n)_{n\in\N}$ is bounded. Thus, there is $\tilde{u}\in\mc{H}$ such that $\tilde{u}_{n_m}\rightharpoonup \tilde{u}$ as $m\to\infty$. We now proceed in three steps: 

\smallskip

First claim: $J'(\tilde{u})=0$. Let $k\in\Z_{\odd}$. By Lemma~\ref{ueberpruefung_cp} it is enough to check $J'(\tilde u)[\tilde v]=0$ for $\tilde v=(v_k\delta_{lk})_{l\in Z_{\odd}}\in \mc{H}_{k,\mono}$ where $v_k\in C_c^\infty(\R)$ by definition of $\mc{H}_{k,\mono}$. For such $\tilde v$ we conclude first by weak convergence that 
$$
J_0'(\tilde u_n)[\tilde v]= B(\tilde u_n,\tilde v) \to B(\tilde u,\tilde v)=J_0'(\tilde u)[\tilde v] \mbox{ as } n\to \infty.
$$
Next, due to the compact support property of $\mc{S}(\tilde v)(x,t)= v_k(x)e^{ik\omega t}$ and the compact embedding $\mc{S}:\mc{H}\hookrightarrow L^{p+1}(K)$, $1<p<p^\ast$ for any compact subset $K\subset\R^2$, cf. Lemma~\ref{HauptresultatSection5.5}, we obtain
\begin{align*}
J_1'(\tilde{u}_n)[\tilde{v}]&= \frac{1}{T} \int_D \vert\mc{S}\tilde{u}_n\vert^{p-1} \mc{S}\tilde{u}_n \mc{S}\tilde{v} d(x,t) \to J_1'(\tilde{u})[\tilde{v}] \mbox{ as } n\to \infty.
\end{align*}
Combining the two convergence results we deduce $J'(\tilde{u})=0$. Note that this chain of arguments only uses that $(\tilde u_n)_{n\in\N}$ is a Palais-Smale sequence for $J$ and not $\tilde u_n\in \mc{M}$.

\smallskip

Second claim: Here we show the existence of a new Palais-Smale sequence $(\tilde v_n)_{n\in\N}$ such that $J(\tilde v_n)\to \inf_{\mc{M}}J$ and that its weak limit $\tilde v$ belongs to $\mc{M}$ (we do not claim that $\tilde v_n\in \mc{M}$). For this purpose we can repeat Steps 2 and 3 from the proof of Lemma~\ref{NichtLin5}. First we obtain that $\mc{S}\tilde u_n^+$ does not converge to $0$ in $L^{p+1}(D)$. From this we obtain (via Lemma~\ref{ConccompLemma}) that 
\begin{align} \label{tildeuneqnulleins}
\liminf_{n\to\infty} \sup_{z\in D} \int_{B_1(z)} \vert \mc{S}\tilde{u}_{n}^+\vert^2 d(x,t)>0.
\end{align}
Therefore we find $\delta>0$, a sequence $(y_n)_{n\in\N}$ in $D$ and a subsequence of $(\tilde{u}_n)_{n\in\N}$ (again denoted by $(\tilde{u}_n)_{n\in\N}$) such that 
\begin{align} \label{Beinsyn_again}
\int_{B_1(y_n)} \vert\mc{S}\tilde{u}^+_n\vert^2 d(x,t) \geq \delta>0 \text{ for all } n\in\N.
\end{align}
Having $y_n=(x_n,t_n)^T$ with $x_n=2\pi m_n+r_n$ for some $m_n\in\Z, r_n\in [0,2\pi)$, we set 
\begin{align*}
\tilde{v}_n(\cdot) \coloneqq \tilde{u}_n(\cdot +2\pi m_n).
\end{align*}
and obtain that $(\tilde v_n)_{n\in \N}$ is again a Palais-Smale sequence for $J$ with $\lim_{n\to\infty} J(\tilde v_n)=\inf_{\mc{M}} J$ and (as in Step 3 of Lemma~\ref{NichtLin5}) with $\tilde{B}\coloneqq [-1,2\pi+1]\times [-1,T+1]$ that
\begin{align*}
\int_{\tilde{B}} \vert\mc{S}\tilde{v}_n^+\vert^2 d(x,t) \geq \delta>0 \text{ for all } n\in\N.
\end{align*}
By making us of the compact embedding to $L^2(\tilde B)$ from Theorem~\ref{HauptresultatSection5.5} up to a subsequence we find that $\tilde v_n\rightharpoonup \tilde v\in \mc{H}$ as $n\to \infty$ with $\tilde v\not=0$. The property $J'(\tilde{v})=0$ follows from the first claim. It remains to show $\tilde{v}^+\neq 0$. Assume by contradiction that $\tilde{v}^+=0$, i.e., $\tilde{v}=\tilde{v}^-$. By testing $J'(\tilde{v})=0$ with $\tilde{v}$ we infer
\begin{align*}
- \Vert \tilde{v}^-\Vert_{\mc{H}}^2 = \frac{1}{T}\int_D f(x,\mc{S}\tilde{v})\tilde v\, d(x,t),
\end{align*} 
a contradiction since the two expressions have different signs. Thus, $\tilde{v}\in \mc{M}$.

\smallskip

Third claim: $\tilde{v}$ minimizes $J$ on $\mc{M}$. Since $\tilde v\in \mc{M}$ we obviously have $J(\tilde{v})\geq \inf_{\mc{M}} J$. Since for a suitable subsequence $\mc{S}(\tilde v_n)\to \mc{S}(\tilde v)$ pointwise a.e. on $D$ the reverse inequality follows from $\frac{1}{2}f(x,s)s-F(x,s)\geq 0$ (cf. Lemma~\ref{FormelNichtlin1}(a)) and Fatou's Lemma as follows:
\begin{align*}
\inf_{\mc{M}} J &= \lim_{n\to \infty} J(\tilde v_n)- \frac{1}{2} J'(\tilde v_n)[\tilde v_n] = 
\lim_{n\to \infty} \int_D \frac{1}{2}f(x,\mc{S}(\tilde v_n))\mc{S}(\tilde v_n)-F(x,\mc{S}(\tilde v_n))\,d(x,t) \\
& \geq  \int_D \frac{1}{2} f(x,\mc{S}(\tilde v))-F(x,\mc{S}(\tilde v))\,d(x,t) = J(\tilde v)-\frac{1}{2} J'(\tilde v)\tilde v = J(\tilde v).
\end{align*}
\qed

\begin{Remark} Let us explain how the case of ''$-$'' in \eqref{Einl2scalar}$_\pm$ can be treated. In this case one keeps the functional $J_1$ but replaces $J_0$ by $-J_0$ and flips the spaces $\mc{H}^+$ and $\mc{H}^-$. Since $J_0$ is an indefinite functional this is without relevance for the proof strategy. All proofs of this section can be carried over with no change.
\end{Remark}


It remains to give the proof of Theorem~\ref{Hauptresultat} and Corollary~\ref{CorzuHauptresultat}. We only do the ''$+$''-case.

\smallskip

\noindent
\textit{Proof of Theorem~\ref{Hauptresultat} and Corollary~\ref{CorzuHauptresultat}:} Let $\tilde{u}$ be a ground state of $J$ obtained previously in Theorem~\ref{MinThm}. The property that $u=\mc{S}\tilde u$ is a weak solution of \eqref{Einl2scalar}$_+$ in the sense of Definition~\ref{Defveryweaksol1} follows from Corollary~\ref{CorzuHauptresultat} if we verify that 
\begin{equation}  \label{for_weak_sol}
 \int_D V(x) u \phi_{tt} + u_x \phi_x \, d(x,t) + \int_D q(x) u \phi\,d(x,t) = \int_D f(x,u) \phi\, d(x,t)
\end{equation}
holds for all $\phi$ from Corollary~\ref{CorzuHauptresultat}. By Theorem~\ref{HauptresultatSection5.5} we have the integrability property $u\in L^{p+1}(D)$. The boundedness of the operator $I: \mc{H}\to \hat H$ from Lemma~\ref{embedding}, the statement preceeding this lemma and the values of $\gamma$ from Theorem~\ref{Abschnachschlag} and $\delta$ from Theorem~\ref{ZielAbschnittNonlinearity} imply the regularity statement for $u$ as stated in Corollary~\ref{CorzuHauptresultat}. This implies in particular all integrability and regularity properties required in Definition~\ref{Defveryweaksol1}.

\medskip

In the following we fix a real-valued test function $\phi=\sum_{k\in\Z}\phi_k(x) e^{ik\omega t}$ with finitely many nonzero coefficient functions $\phi_k\in C_c^\infty(\R)$. Using that $\tilde u$ is a critical point of the functional $J$ from Theorem~\ref{MinThm} together with Lemma~\ref{ueberpruefung_cp} we obtain 
\begin{equation} \label{for_weak_sol_fourier}
\sum_{k\in \Z} \int_\R \left(V(x)k^2+q(x)\right) u_k(x)\phi_k(x) + u_k'(x)\phi_k'(x)\,dx = \frac{1}{T}\int_D f(x,u(x,t)) \phi(x,t)\,d(x,t).
\end{equation}
Here we have used for $k$ even that $u_k=0$ and $f(x,u(x,\cdot))_k = \frac{1}{T}\int_0^T f(x,u(x,t)e^{-ik\omega t}\,dt=0$ due to (H3). Notice that \eqref{for_weak_sol_fourier} is just \eqref{for_weak_sol} for our particular test function $\phi$. Here and in the following we understand in case (V1) the integral $\int_\R \delta_{\per}(x) u_k\phi_k \,dx$ as a symbol for $\sum_{n\in\Z} u_k(2\pi n)\phi_k(2\pi n)$.  

\medskip

It remains to show that the assumption of having only finitely many nonzero compactly supported coefficient functions $\phi_k\in C_c^\infty(\R)$ in the definition of $\phi=\sum_{k\in\Z}\phi_k(x) e^{ik\omega t}$ may be relaxed in favor of $\phi\in H^{\tilde \alpha}(0,T;H^1(\R))\cap H^{\tilde\beta}(0,T;L^2(\R))$ with $\tilde\alpha, \tilde\beta>0$ as in Corollary~\ref{CorzuHauptresultat}. The result will follow from the first part of the theorem by letting the summation index in the definition of $\phi$ tend to infinity and using the following estimates explained first in the cases (V2), (V3):
\begin{align}
\left| \int_D f(x,u) \phi\,d(x,t)\right| & \leq C( \|u\|_{L^2(D)}\|\phi\|_{L^2(D)}+ \|u\|_{L^{p+1}(D)}^p \|\phi\|_{L^{p+1}(D)}), \label{ab1}\\
\left| \sum_{k} \int_\R k^2 u_k\phi_k \,dx \right| & \leq \|u\|_{H^\beta(0,T;L^2(\R)} \|\phi\|_{H^{\tilde\beta}(0,T;L^2(\R)}, \label{ab2}\\
\left| \sum_{k} \int_\R u_k'\phi_k' \,dx \right| & \leq \|u\|_{H^\alpha(0,T;H^1(\R)} \|\phi\|_{H^{\tilde\alpha}(0,T;H^1(\R)}. \label{ab3} 
\end{align}
The first estimate \eqref{ab1} follows from H\"older's inequality since assumptions (H1), (H2) imply the estimate $|f(x,s)|\leq \epsilon |s|+ C_\epsilon |s|^p$. Since $\tilde\alpha\geq \alpha$ and $\tilde\beta\geq \beta$ we get $\phi\in L^{p+1}(D)$. The second estimate \eqref{ab2} is a consequence of the Cauchy-Schwarz inequality and $\tilde\beta \geq 2-\beta$, i.e.,
$$
\left| \sum_{k} \int_\R k^2 u_k\phi_k \,dx \right| \leq \left(\sum_k \|u_k\|_{L^2(\R)}^2 |k|^{2\beta} \right)^{1/2}\left(\sum_k \|\phi_k\|_{L^2(\R)}^2 |k|^{4-2\beta} \right)^{1/2} =\|u\|_{H^\beta(0,T;L^2(\R)} \|\phi\|_{H^{2-\beta}(0,T;L^2(\R)}.
$$
Finally, the third estimate \eqref{ab3} is also a consequence of the Cauchy-Schwarz inequality and $\tilde\alpha\geq -\alpha$, i.e.,
$$
\left| \sum_{k} \int_\R u_k'\phi_k' \,dx \right| \leq \left(\sum_k \|u_k\|_{H^1(\R)}^2 |k|^{2\alpha} \right)^{1/2}\left(\sum_k \|\phi_k\|_{H^1(\R)}^2 |k|^{-2\alpha} \right)^{1/2} =\|u\|_{H^\alpha(0,T;H^1(\R)} \|\phi\|_{H^{-\alpha}(0,T;H^1(\R)}.
$$
In case (V1) only the estimate \eqref{ab2} looks different: here we need to show that 
\begin{equation}
\left| \sum_{k} \int_\R \delta_{\per}(x) k^2 u_k\phi_k \,dx \right|^2 \leq C(\|u\|_{H^\alpha(0,T;H^1(\R))}^2+\|u\|_{H^\beta(0,T;L^2(\R)}^2)(\|\phi\|_{H^{\tilde\alpha}(0,T;H^1(\R))}^2+\|\phi\|_{H^{\tilde\beta}(0,T;L^2(\R)}^2).\label{ab2strich}
\end{equation}
To see this note first that the Cauchy-Schwarz inequality and Lemma~\ref{fundablpktwgegennull} allow to estimate the left-hand side of \eqref{ab2strich} by
\begin{align*}
&\left| \sum_{k} \int_\R \delta_{\per}(x) k^2 u_k\phi_k \,dx \right|^2 \leq \left(\sum_k\sum_n |k|^{2a} |u_k(2\pi n)|^2\right) \left(\sum_k\sum_n |k|^{4-2a} |\phi_k(2\pi n)|^2\right) \\
&\leq \left(\sum_k |k|^{2a} \left(\frac{1}{2\pi}+\frac{1}{2\varepsilon}\right) \|u_k\|_{L^2(\R)}^2 + |k|^{2a}\frac{\varepsilon}{2}\|u_k'\|_{L^2(\R)}^2\right) \left(\sum_k |k|^{4-2a} \left(\frac{1}{2\pi}+\frac{1}{2\tilde \varepsilon}\right) \|\phi_k\|_{L^2(\R)}^2 + |k|^{4-2a}\frac{\tilde\varepsilon}{2}\|\phi_k'\|_{L^2(\R)}^2\right)
\end{align*}
for arbitrary $\varepsilon, \tilde\varepsilon>0$ and $a\in\R$. We choose $\varepsilon = |k|^{-3-2a}$, $\tilde\varepsilon = |k|^{2a-4+2\tilde\alpha}$. Therefore 
\begin{align*}
&\left| \sum_{k} \int_\R \delta_{\per}(x) k^2 u_k\phi_k \,dx \right|^2 \\
& \leq C\left(\sum_k (|k|^{2a}+|k|^{4a+3}) \|u_k\|_{L^2(\R)}^2 + |k|^{-3}\|u_k'\|_{L^2(\R)}^2\right) \left(\sum_k (|k|^{4-2a}+|k|^{8-4a-2\tilde\alpha}) \|\phi_k\|_{L^2(\R)}^2 + |k|^{2\tilde\alpha}\|\phi_k'\|_{L^2(\R)}^2\right)
\end{align*}
Considering the regularity of $u$ it turns out that the optimal value is $a=-1/2$. The assumptions $4-2a=5\leq 2\tilde\beta$ and $8-4a-2\tilde\alpha=10-2\tilde\alpha\leq 2\tilde\beta$ imply that the right hand side is controlled as claimed in \eqref{ab2strich}. Since \eqref{ab3} remains the same, we see that also $\tilde\alpha\geq -\alpha=3/2$ is needed. This finishes the proof of case (V1). \qed

\section{Proof of boundedness of $\mc{S}$} \label{sec:proof_for_S}
We split the proof of Theorem~\ref{HauptresultatSection5.5} into several steps and make use of the following intermediate space. Let 
\begin{align*}
\hat{H} &\coloneqq \Big\{\tilde u=(u_k)_{k\in\Z_{\odd}}: u_k \in H^1(\R) \text{ for all } k\in\Z_{\odd} \text{ s.t. } \Vert (u_k)_{k\in\Z_{\odd}}\Vert_{\hat{H}}<\infty \Big\},  \\
\Vert \tilde u\Vert_{\hat{H}}^2 &\coloneqq \sum_{k\in\Z_{\odd}} \paren{\abs{k}^\gamma \Vert u_k\Vert^2_{L^2(\R)}+ |k|^\delta\Vert u_k'\Vert^2_{L^2(\R)}},
\end{align*}
with $\gamma$ as in Theorem~\ref{Abschnachschlag} and $\delta$ as in Theorem~\ref{ZielAbschnittNonlinearity}. Note that $\hat{H}$ is isometrically isomorphic to $H^{\delta/2}(0,T; H^1(\R))\cap H^{\gamma/2}(0,T; L^2(\R))$. 

\begin{Lemma} The embedding $I: \mc{H}\to \hat H$ is bounded. \label{embedding}
\end{Lemma}

\begin{proof} By the construction of norms in $\hat H$, $\mc{H}$ and Theorem~\ref{Abschnachschlag}, Theorem~\ref{ZielAbschnittNonlinearity} we see that 
$$
\|\tilde u\|_{\hat H}^2 \leq C \sum_{k\in \Z_{\odd}} b_{|L_k|}(u_k,u_k) = C\|\tilde u\|_{\mc{H}}^2
$$
for all $\tilde u\in \mc{H}$ with a constant $C>0$ which is independent on $\tilde u$. 
\end{proof}

For $2\leq q\leq\infty$ and $u\in L^q(D)\cap L^1(D)$ let us denote by $\hat u_k(\xi)$ the Fourier-transform with respect to $(x,t)\in D=\R\times(0,T)$, i.e.,
$$
\hat u_k(\xi) = \frac{1}{T\sqrt{2\pi}}\int_D u(x,t) e^{-i(\xi x+\omega kt)}\,d(x,t), \quad \omega= \frac{2\pi}{T}.
$$
For functions $\hat u = (\hat u_k(\xi))_{\xi\in \R, k\in \Z}\in L^{q'}(\R\times\Z)$ we consider the space $L^{q'}(\R\times\Z)$ with the norm $\|\hat u\|_{L^{q'}} = (\sum_{k\in \Z} \int_\R |\hat u_k(\xi)|^{q'}\,d\xi)^{1/{q'}}$ if $1\leq q'\leq 2$. The next result is a Riesz-Thorin based Hausdorff-Young inequality.

\begin{Lemma} \label{hausdorff_young}
Let $1\leq q'\leq 2$. Then there exists a constant $C(q)$ such that 
$$
\|u\|_{L^q(D)} \leq C(q) \|\hat u\|_{L^{q'}(\R\times\Z)}
$$
for all $\hat u \in L^{q'}(\R\times\Z)$. Hence the inverse Fourier-transform $\mc{S}_1: \hat u \mapsto u=\frac{1}{\sqrt{2\pi}}\sum_{k\in \Z} \int_\R \hat u_k(\xi)e^{i(\xi x+\omega kt)}\,d\xi$ has a continuous extension $L^{q'}(\R\times \Z)$ to $L^{q}(D)$. 
\end{Lemma}

\begin{proof} By the Riesz-Thorin theorem it suffices to check the extremal cases $q'=1$ and $q'=2$. For $q'=2$ we have the Plancherel identities 
$$
\|u\|^2_{L^2(D)} =\int_D |u(x,t)|^2\,d(x,t) = \sum_{k\in\Z} \int_\R |u_k(x)|^2\,dx = \sum_{k\in\Z} \int_\R |\hat u_k(\xi)|^2\,d\xi = \|\hat u\|^2_{L^2(\R\times\Z)}
$$
and for $q'=1$ we have 
$$
\|u\|_{L^\infty(D)} = \sup_{(x,t)\in D} \left|\frac{1}{\sqrt{2\pi}}\sum_{k\in\Z} \int_\R \hat u_k(\xi)e^{i(\xi x+\omega kt)}\,d\xi\right| \leq \frac{1}{\sqrt{2\pi}}\sum_{k\in \Z} \int_\R |\hat u_k(\xi)|\,d\xi = \frac{1}{\sqrt{2\pi}}\|\hat u\|_{L^1(\R\times T)}.
$$
\end{proof}

\begin{Lemma} \label{embedding_by_fourier}
For $\tilde u\in \hat H$ let $\mc{S}_2(\tilde u) = \hat u$ with $\hat u_k$ being the $L^2$-extension of $\hat u_k(\xi) := \frac{1}{\sqrt{2\pi}}\int_\R u_k(x) e^{-ix\xi}\,dx$. Then $\mc{S}_2: \hat H\to L^{q'}(\R\times\Z)$ is a bounded linear operator for $2\geq q'> {q^\ast}'$ and 
$$
{q^\ast}' =\left\{\begin{array}{ll}
\frac{3}{2} & \mbox{ in case (V1)}, \vspace{\jot}\\
\frac{4}{3} & \mbox{ if case (V2)}, \vspace{\jot}\\
\frac{4}{2+\gamma} & \mbox{ in case (V3)}.
\end{array}
\right.
$$
Moreover, if we consider $\mc{S}_2^{k_0}: \tilde u  \mapsto (\ldots,0,\hat u_{-k_0}, \hat u_{-k_0+1}, \ldots, \hat u_{k_0-1}, \hat u_{k_0},0,\ldots)$ then $\mc{S}_2 = \lim_{k_0\to \infty} \mc{S}_2^{k_0}$ in the operator norm. 
\end{Lemma}

\begin{proof} Choose $\rho\geq 2>\gamma>0$. We note that Young's inequality with exponents $ \frac{\rho}{\rho-\gamma}$ and $\frac{ \rho}{\gamma}$ implies that 
$$
|\xi|^\frac{2\gamma}{\rho} |k|^{\rho-\gamma} \leq \frac{\rho-\gamma}{\rho}|k|^\rho + \frac{\gamma}{\rho} |\xi|^2 \mbox{ for all } \xi\in \R, k\in \Z_{\odd}.
$$
Therefore
\begin{equation} \label{vorauss_fuer_einbettung}
|\xi|^\frac{2\gamma}{\rho} +|k|^\gamma \leq C(\rho,\gamma) \left(|k|^\gamma + |\xi|^2 |k|^{\gamma-\rho}\right) \mbox{ for all } \xi\in \R, k\in \Z_{\odd}.
\end{equation}
Making use of this elementary inequality we deduce for $\tilde u\in \hat{H}$ that 
\begin{align}
\|\hat u\|^{q'}_{L^{q'}(\R\times\Z_{\odd})} & = \sum_{k\in \Z_{\odd}} \int_\R |\hat u_k(\xi)|^{q'}\,d\xi = \sum_{k\in \Z_{\odd}} \int_\R |\hat u_k(\xi)|^{q'}\frac{(|\xi|^\frac{2\gamma}{\rho}+|k|^\gamma)^\frac{q'}{2}}{(|\xi|^\frac{2\gamma}{\rho}+|k|^\gamma)^\frac{q'}{2}}\,d\xi \nonumber\\
& \leq \left(\sum_{k\in \Z_{\odd}} \int_\R |\hat u_k(\xi)|^2 (|\xi|^\frac{2\gamma}{\rho}+|k|^\gamma) \,d\xi\right)^\frac{q'}{2}\cdot \left(\sum_{k\in \Z_{\odd}} \int_\R (|\xi|^\frac{2\gamma}{\rho}+|k|^\gamma)^\frac{-q'}{2-q'}\,d\xi\right)^\frac{2-q'}{2} \nonumber \\
& \leq \tilde C(\rho,\gamma) \left(\sum_{k\in \Z_{\odd}} \int_\R |\hat u_k(\xi)|^2 (|k|^\gamma + |\xi|^2|k|^{\gamma-\rho}) \,d\xi\right)^\frac{q'}{2}\cdot \left(\sum_{k\in \Z_{\odd}} \int_\R (|\xi|^\frac{2\gamma}{\rho}+|k|^\gamma)^\frac{-q'}{2-q'}\,d\xi\right)^\frac{2-q'}{2} \label{estimate_s2u}\\
& = \tilde C(\rho,\gamma) \left(\sum_{k\in \Z_{\odd}} \|u_k\|^2_{L^2(\R)} |k|^\gamma + \|u_k'\|^2_{L^2(\R)} |k|^{\gamma-\rho}\right)^\frac{q'}{2}\cdot \left(\sum_{k\in \Z_{\odd}} \int_\R (|\xi|^\frac{2\gamma}{\rho}+|k|^\gamma)^\frac{-q'}{2-q'}\,d\xi\right)^\frac{2-q'}{2} \nonumber\\
& = \tilde C(\rho,\gamma) I_1^\frac{q'}{2} \cdot I_2^\frac{2-q'}{2}. \nonumber
\end{align}
Next we investigate the expressions $I_1$ and $I_2$ separately. To check the convergence of $I_2$ we compute
\begin{align*}
I_2 &= \sum_{k\in \Z_{\odd}} \int_\R (|\xi|^\frac{2\gamma}{\rho}+ |k|^\gamma)^\frac{-q'}{2-q'}\,d\xi = \sum_{k\in Z_{\odd}} |k|^{-\frac{q'\gamma}{2-q'}} \int_\R \left(1+ \frac{|\xi|^\frac{2\gamma}{\rho}}{|k|^\gamma}\right)^\frac{-q'}{2-q'}\,d\xi \\
&= \sum_{k\in \Z_{\odd}} |k|^{\frac{-q'\gamma}{2-q'}+\frac{\rho}{2}} \int_\R \left(1+ |\tau|^\frac{2\gamma}{\rho}\right)^\frac{-q'}{2-q'}\,d\tau \mbox{ with } \xi = |k|^\frac{\rho}{2}\tau.
\end{align*}
The integral converges provided $\frac{-2\gamma q'}{\rho(2-q')}<-1$, i.e., $q'> \frac{2\rho}{2\gamma+\rho}$. The series converges provided $\frac{-q'\gamma}{2-q'}+\frac{\rho}{2}<-1$, i.e., $q'> \frac{2\rho+4}{2\gamma+\rho+2}$. The more restrictive condition amounts to $q'> \frac{2\rho+4}{2\gamma+\rho+2}$.

Now we turn to estimating $I_1$. In case of assumption (V1) where $\gamma=1$, $\delta=-3$ the choice $\rho=4$ leads to 
$I_1 \leq C\|\tilde u\|_{\hat H}^2$ and the convergence condition $q'>\frac{2\rho+4}{2\gamma+\rho+2} = \frac{3}{2} = {q^\ast}'$. In case of assumption (V2) where $\gamma=1$, $\delta=-1$ the choice $\rho=2$ amounts to $I_1 \leq C\|\tilde u\|_{\hat H}^2$ and $q'>\frac{2\rho+4}{2\gamma+\rho+2}=\frac{4}{3}={q^\ast}'$. Finally in the case of assumption (V3) where $\gamma<\frac{3}{2}-r$, $\delta=\gamma-2$ the choice $\rho=2$ leads to $q'>\frac{2\rho+4}{2\gamma+\rho+2}=\frac{4}{2+\gamma}={q^\ast}'$. Thus, the convergence of $I_2$ is ensured in any case. Therefore, in all cases we have found that $I_2$ converges for $q'>{q^\ast}'$ and that $I_1\leq C \|\tilde u\|^2_{\hat H}$. In view of \eqref{estimate_s2u} this establishes the claim of the boundedness of $\mc{S}_2$.

\medskip

The statement that $\mc{S}_2 = \lim_{k_0\to \infty} \mc{S}_2^{k_0}$ in the operator norm can be seen as follows: as in \eqref{estimate_s2u} the difference $\|(\mc{S}_2-\mc{S}_2^{k_0})\hat u\|^{q'}_{L^{q'}(\R\times\Z_{\odd})}$ can be estimated by $\tilde C(\rho,\gamma) I_1^\frac{q'}{2} \cdot I_{2,k_0}^\frac{2-q'}{2}$ where 
$$
I_{2,k_0} = \sum_{|k|> k_0, k\in \Z_{\odd}} |k|^{\frac{-q'\gamma}{2-q'}+\frac{\rho}{2}} \int_\R \left(1+ |\tau|^\frac{2\gamma}{\rho}\right)^\frac{-q'}{2-q'}\,d\tau \to 0 \mbox{ as } k_0 \to \infty.
$$
\end{proof}

After these preparations the proof of Theorem~\ref{HauptresultatSection5.5} becomes quite simple.

\smallskip

\textit{Proof of Theorem~\ref{HauptresultatSection5.5}}: Observe that $\mc{S}=\mc{S}_1\circ \mc{S}_2\circ I$, where the operators $\mc{S}_1$, $\mc{S}_2$, $I$ are bounded by Lemma~\ref{embedding}, Lemma~\ref{hausdorff_young}, Lemma~\ref{embedding_by_fourier}, respectively. The condition $2\geq q'> {q^\ast}'$ is equivalent to $2\leq q < q^\ast$. 

\smallskip

To see that $\mc{S}\colon\mc{H}\to L^{p+1}(D)$ is one-to-one let $\tilde{u}\in\mc{H}$ be given with $\mc{S}\tilde{u}=0$. In particular, $\mc{S}\tilde{u}\in L^2(D)$ and hence by the Plancherel identity
\begin{align*}
0=\Vert\mc{S}\tilde{u}\Vert_{L^2(D)}^2=\sum_{k\in\Z_{\odd}}\int_\R \vert u_{k}(x)\vert^2\,dx
\end{align*}
we get $\tilde{u}=0$. 

\smallskip

Let now $K\subset \overline{D}$ be compact. The operator $\mc{S}^{k_0} = \mc{S}_1\circ \mc{S}_2^{k_0}\circ I$ sees only finitely many Fourier-coefficients in time and these coefficients belong to $H^1(\R)$ with respect to space. Hence, by restriction to $K$ the operator $\mc{S}^{k_0}$ maps $\mc{H}$ compactly into $L^q(K)$ for every $q\in [1,\infty)$. For $q$ in the range of Theorem~\ref{HauptresultatSection5.5} we can use Lemma~\ref{embedding_by_fourier} to see that $\mc{S}^{k_0}$ converges to $\mc{S}$ in the operator norm as $k_0\to\infty$, and hence the compactness of $\mc{S}: \mc{H} \to L^q(K)$ follows. 
\qed



\section{Appendix} \label{ZweiterAnhang}

\begin{Lemma} \label{monotonicity}
Let $V_1, V_2, V\in H^1_{\per}(\R)$ be periodic potentials with $V_1\leq V_2$ and $\essinf_\R V_1>0$. Consider the differential operators $L_i := \frac{1}{V_i} \bigl(-\frac{d^2}{dx^2} + V(x)\bigr)$, $i=1,2$ and $L= -\frac{d^2}{dx^2} + V(x)$. Then the following holds for the resolvent sets $\rho(L_1), \rho(L_2), \rho(L)$: 
\begin{itemize}
\item[(a)] If $[-a,a]\subset \rho(L_2)$ then $[-a,a]\subset \rho(L_1)$.
\item[(b)] If $[-a,a]\subset \rho(L_2)$ and $V_1\equiv\const$ then $[-V_1a, V_1a]\subset \rho(L)$.
\end{itemize}
\end{Lemma}

\begin{proof} Band edges of periodic differential operators are occupied by either periodic or anti-periodic eigenvalues, cf. \cite{EasthamPer}. If $\lambda(L_i)$ denotes the $n$-th periodic eigenvalue of $L_i$ ($i=1,2$) and $\lambda(L_2)>0$ then the Poincar\'{e} min-max principle implies that $\lambda(L_1) \geq \lambda(L_2)$. If $\lambda(L_2)<0$ then $\lambda(L_1)\leq \lambda(L_2)$. The same holds for the antiperiodic eigenvalues. Therefore, if $L_2$ has a spectral gap around $0$ then $L_1$ also has a spectral gap near zero of at least the same size. This proves (a). Statement (b) follows from (a) and $\sigma(L_1)= \frac{1}{V_1}\sigma(L)$ so that $\rho(L_1)= \frac{1}{V_1}\rho(L)$.  
\end{proof}

\begin{Lemma} \label{fundablpktwgegennull}
Let $f\in H^1(\R)$. Then for $\varepsilon>0$ we have
\begin{align} \label{AbschfuerweiterePunktauswertung}
\sum_{n\in\Z} \vert f(2\pi n)\vert^2 \leq \paren{\frac{1}{2\pi}+\frac{1}{2\varepsilon}} \Vert f\Vert^2_{L^2(\R)} + \frac{\varepsilon}{2} \Vert f'\Vert^2_{L^2(\R)}.
\end{align}
\end{Lemma}

\begin{proof}
Let $u_n(x)\coloneqq f(2\pi n+x)$. We compute
\begin{align} \label{Teil1derAbschfueru0}
|u_n(0)|^2=\frac{1}{\pi}\int_{-\pi}^0 \frac{d}{dt}\bra{(t+\pi)|u_n(t)|^2}dt \leq \frac{1}{\pi} \int_{-\pi}^0 |u_n(t)|^2 dt +2\int_{-\pi}^0 \vert u_n(t)u_n'(t)\vert dt.
\end{align}
In the same manner
\begin{align} \label{Teil2derAbschfueru0}
|u_n(0)|^2 = -\frac{1}{\pi}\int_0^\pi \frac{d}{dt}\bra{(\pi-t)|u_n(t)|^2}dt \leq \frac{1}{\pi}\int_0^\pi |u_n(t)|^2 dt+ 2\int_0^\pi \vert u_n(t)u_n'(t)\vert dt.
\end{align}
By adding \eqref{Teil1derAbschfueru0} and \eqref{Teil2derAbschfueru0} we conclude
\begin{align*}
|u_n(0)|^2\leq \frac{1}{2\pi} \Vert u_n\Vert^2_{L^2(-\pi,\pi)}+\Vert u_n\Vert_{L^2(-\pi,\pi)}\Vert u_n'\Vert_{L^2(-\pi,\pi)} \leq \frac{1}{2}\paren{\frac{1}{\varepsilon}+\frac{1}{\pi}}\Vert u_n\Vert^2_{L^2(-\pi,\pi)}+\frac{\varepsilon}{2}\Vert u_n'\Vert^2_{L^2(-\pi,\pi)}
\end{align*}
and hence 
\begin{align*} 
|f(2\pi n)|^2\leq \frac{1}{2}\paren{\frac{1}{\varepsilon}+\frac{1}{\pi}} \Vert f(2\pi n+\cdot)\Vert^2_{L^2(-\pi,\pi)}+\frac{\varepsilon}{2}\Vert f'(2\pi n+\cdot)\Vert^2_{L^2(-\pi,\pi)}.
\end{align*}
The claim follows by a summation over $n\in\Z$.
\end{proof}

\subsection*{Proof of Lemma~\ref{ueberpruefung_cp}}~ For the purpose of this proof let us define the space 
$$
\mc{H}_0 \coloneqq \Bigl\{ \tilde{u}=(u_k)_{k\in \Z_{\odd}}: u_k\in H^1(\R) \text{ for all } k\in\Z_{\odd} \text{  and } \sum_{k\in\Z_{\odd}} \int_\R  |\lambda| \langle P_\lambda^k u_k,u_k\rangle_{L^2(\R)}\,d\lambda <\infty \Bigr\}
$$
equipped with the same norm and inner product as $\mc{H}$. 
It can be seen as a variant of $\mc{H}$ but without the additional requirement of conjugation-symmetry $\bar u_k = u_{-k}$. Clearly, $\mc{H}_{k,\mono}\not\subset \mc{H}$ but $\mc{H}_{k,\mono}\subset \mc{H}_0$. 

First we check that $J'(\tilde u)=0$ implies (and hence is equivalent to) $J'(\tilde u)[\tilde\phi]=0$ for all $\tilde\phi\in \mc{H}_0$, i.e., that we can allow test functions $\tilde\phi$ without the extra conjugation-symmetry. For $\tilde\phi\in \mc{H}_0$ let us define the splitting  
$$
\phi_k = \phi^a_k+\phi^b_k \quad \mbox{ with } \quad \phi^a_k := \frac{\phi_k+\bar\phi_{-k}}{2}, \quad \phi^b_k := \frac{\phi_k-\bar\phi_{-k}}{2}.
$$
Then $\tilde \phi^a, i\tilde \phi^b\in \mc{H}$ and hence $J'(\tilde u)[\tilde \phi^a]=0$ and $0=J'(\tilde u)[i\tilde\phi^b]=(-i)J'(\tilde u)[\tilde\phi^b]$. Therefore we also have $J'(\tilde u)[\tilde \phi]=J'(\tilde u)[\tilde\phi^a+\tilde\phi^b]=0$ as claimed.

\smallskip

(i) $\Leftrightarrow$ (ii): With the help of the first step we know that $J'(\tilde u)|_{\mc{H}}=0$ implies $J'(\tilde u)|_{\mc{H}_{k,\mono}}=0$. Now we verify the reverse: $J'(\tilde u)|_{\mc{H}_{k,\mono}}=0$ for all $k\in \Z_{\odd}$ implies $J'(\tilde u)|_{\mc{H}}=0$. Note that $\overline{\mc{H}}_{k,\mono}$ consists of all mono-modal Fourier-series, where the only non-vanishing Fourier-coefficient belongs to $H^1(\R)$. Therefore any $\tilde \phi\in \mc{H}_0$ can be seen as $\tilde\phi = \lim_{m\to \infty} \tilde\phi^m$ (convergence with respect to the $\|\cdot\|_{\mc{H}}$-norm) where for $m\in \N$, $m$ odd, we set 
$$
\tilde\phi^m := (\phi^m_l)_{l\in \Z_{\odd}} \mbox{ with } \phi^m_l:= \left\{\begin{array}{ll}
\phi_l & \mbox{ if } l \in \Z_{\odd}, |l|\leq m, \vspace{\jot}\\
0 & \mbox{ if } l\in \Z_{\odd}, |l|> m. 
\end{array}
\right.
$$
Since $\tilde\phi^m$ is a finite sum of members of $\overline{\mc{H}}_{k,\mono}$ for $k=-m, -m+2, \ldots ,-1,1,\ldots, m-2,m$ we have $J'(\tilde u)[\tilde\phi^m]=0$. Then $J'(\tilde u)[\tilde\phi]=0$ follows since $J'$ is a continuous linear functional on $\mc{H}_0$ and $\|\tilde \phi^m-\tilde\phi\|_{\mc{H}}\to 0$ as $m\to \infty$. Thus $J'(\tilde u)|_{\mc{H}_0} =0$ and the claim $J'(\tilde u)|_{\mc{H}} =0$ follows by the first step.\qed

\subsection*{Proof of Lemma~\ref{ConccompLemma}}

\begin{Lemma} \label{UeberdeckmitvierLemma}
Let $0<r<T$. Then there is a sequence $(y_l)_{l\in\N}$ in $D$ s.t. 
$D\subset \bigcup_{l\in\N} B_r(y_l)$ and each point $y\in D$ is contained in at most four balls $B_r(y_l)$.
\end{Lemma}

\begin{proof}
The statement follows if we choose $(y_l)_{l\in\N}$ to be an enumeration of $r \Z^2 \cap D$.
\end{proof}

\begin{Lemma} \label{FolgeLvierdomainbeschr}
With the notation of Lemma~\ref{UeberdeckmitvierLemma} and $p^\ast$ from Theorem~\ref{Hauptresultat} for every $\bar q  \in [2,p^\ast+1)$ there is a constant $C=C(r,\bar q)>0$ such that
\begin{align*}
\sum_{l\in\N} \Vert \mc{S}\tilde{u}\Vert^2_{L^{\bar q}(B_r(y_l))} \leq C \Vert \tilde{u}\Vert^2_{\mc{H}}
\end{align*}
for all $\tilde{u}\in\mc{H}$.
\end{Lemma}

\begin{proof} By the embedding $I:\mc{H}\to \hat H$ form Lemma~\ref{embedding} with $\hat H$ as defined at the beginning of Section~\ref{sec:proof_for_S} it is sufficient to prove 
$$
\sum_{l\in\N} \Vert (\mc{S}_1\circ\mc{S}_2)\tilde{u}\Vert^2_{L^{\bar q}(B_r(y_l))} \leq C_r \Vert \tilde{u}\Vert^2_{\hat H}
$$
for all $\tilde{u}\in\mc{H}$. Due to Lemma~\ref{UeberdeckmitvierLemma} we can distinguish balls $B_r(y_l), l\in\N$ which are completely in $D$ and others which protrude from $D$. However, since the function $(\mc{S}_1\circ\mc{S}_2)\tilde u$ is periodic in the second variable and hence its norm in $L^{\bar q}(B_r(y_l))$ is invariant under translations in $t$-direction, the distinction between these balls it not needed for proving the claimed estimate. We abbreviate $\tilde{D}_r\coloneqq \bigcup_{l\in\N} B_r(y_l)$ with $y_l=(x_l,t_l)$ with $x_l\in r\Z$. Let $\phi_l \in C_c^\infty(\R)$ be a cut-off function with $0 \leq \phi_l\leq 1$, $|\phi_l'|\leq 2/r$, $\supp\phi_l\subset (x_l-2r,x_l+2r)$ and $\phi_l \equiv 1$ on $[x_l-r,x_l+r]$. If we define $\tilde u\phi_l := (u_k\phi_l)_{k\in \Z_{\odd}}$ then clearly $\tilde u\phi_l\in \hat H$. Moreover, since 
$$
\|(\mc{S}_1\circ \mc{S}_2) \tilde u\|_{L^{\bar q}(B_r(y_l))} \leq \|(\mc{S}_1\circ \mc{S}_2) (\tilde u\phi_l)\|_{L^{\bar q}(D)}\leq C\|\tilde u\phi_l\|_{\hat H}
$$
due to Lemma~\ref{hausdorff_young}, Lemma~\ref{embedding_by_fourier}, it remains to show the estimate $\sum_{l\in \N}\|\tilde u\phi_l\|_{\hat H}^2\leq C\|\tilde u\|_{\hat H}^2$. By the chain rule we have $\|(u_k\phi_l)'\|_{L^2(\R)}^2 \leq 2\|u_k'\phi_l\|^2_{L^2(\R)} + 2 \|u_k\phi_l'\|^2_{L^2(\R)}$ and hence the definition of the norm in $\hat H$ (note that $\delta<0 <\gamma$) implies 
\begin{align*}
\sum_{l\in\N} \|\tilde u\phi_l\|_{\hat H}^2 &= \sum_{l\in \N}\sum_{k\in \Z_{\odd}} |k|^\gamma \|u_k\phi_l\|_{L^2(\R)}^2 + |k|^\delta \|(u_k\phi_l)'\|_{L^2(\R)}^2 \\
& \leq \sum_{l\in \N}\sum_{k\in \Z_{\odd}} |k|^\gamma \|u_k\|_{L^2(x_l-2r,x_l+2r)}^2 +  2|k|^\delta \|u_k'\|_{L^2(x_l-2r,x_l+2r)}^2 + \frac{8}{r^2}|k|^\delta \|u_k\|_{L^2(x_l-2r,x_l+2r)}^2 \\
& \leq C_r \sum_{l\in \N}\sum_{k\in \Z_{\odd}} |k|^\gamma \|u_k\|_{L^2(x_l-2r,x_l+2r)}^2 + |k|^\delta \|u_k'\|_{L^2(x_l-2r,x_l+2r)}^2 \\
& \leq 2C_r \sum_{k\in \Z_{\odd}} |k|^\gamma \|u_k\|_{L^2(\R)}^2 + |k|^\delta \|u_k'\|_{L^2(\R)}^2 = 2 C_r \|\tilde u\|_{\hat H}^2.
\end{align*}
which finishes the proof.
\end{proof}

\noindent
\textit{Proof of Lemma~\ref{ConccompLemma}:} W.l.o.g. we may assume $r\in (0,T)$. Fix $\tilde{u}\in\mc{H}$ and $y\in D$. Let $q<\bar q<p^\ast+1$. By H\"older interpolation for $s\in (q,\bar q)$ there is $\lambda=\frac{s-q}{\bar q-q}\cdot\frac{\bar q}{s}$ such that
\begin{align*}
\Vert \mc{S}\tilde{u}\Vert_{L^s(B_r(y))} \leq \Vert \mc{S}\tilde{u}\Vert^{1-\lambda}_{L^q(B_r(y))} \Vert \mc{S}\tilde{u}\Vert^\lambda_{L^{\bar q}(B_r(y))}.
\end{align*}
For $s=2+q\frac{\bar q-2}{\bar q}$ we have $\lambda=\frac{2}{s}$ and in particular 
\begin{align} \label{Gleichungfuerfesty}
\Vert \mc{S}\tilde{u}\Vert^s_{L^s(B_r(y))}\leq \Vert \mc{S}\tilde{u}\Vert^{(1-\lambda)s}_{L^q(B_r(y))} \Vert \mc{S}\tilde{u}\Vert^2_{L^{\bar q}(B_r(y))} \leq \Vert \mc{S}\tilde{u}\Vert^2_{L^{\bar q}(B_r(y))} \sup_{z\in D} \Vert \mc{S}\tilde{u}\Vert^{(1-\lambda)s}_{L^q(B_r(z))}.
\end{align}
We now choose the sequence $(y_l)_{l\in\N}$ from Lemma~\ref{UeberdeckmitvierLemma}, then use \eqref{Gleichungfuerfesty} for $y=y_l$ and perform a summation over $l\in\N$. Due to Lemma~\ref{UeberdeckmitvierLemma} we obtain
\begin{align*}
\Vert \mc{S}\tilde{u}\Vert^s_{L^s(D)} \leq \sum_{l\in\N} \Vert \mc{S}\tilde{u}\Vert^s_{L^s(B_r(y_l))} \leq \sum_{l\in\N} \Vert \mc{S}\tilde{u}\Vert^2_{L^{\bar q}(B_r(y_l))} \sup_{z\in D} \Vert \mc{S}\tilde{u}\Vert^{(1-\lambda)s}_{L^q(B_r(z))}.
\end{align*}
Lemma~\ref{FolgeLvierdomainbeschr} guarantees the existence of $C=C(r,\bar q)>0$ s.t. $\sum_{l\in\N} \Vert \mc{S}\tilde{u}\Vert^2_{L^{\bar q}(B_r(y_l))} \leq C \Vert \tilde{u}\Vert^2_{\mc{H}}$. Thus,
\begin{align} \label{conccompsummary}
\Vert \mc{S}\tilde{u}\Vert^s_{L^s(D)} \leq C \Vert\tilde{u}\Vert^2_{\mc{H}} \sup_{z\in D} \Vert \mc{S}\tilde{u}\Vert^{(1-\lambda)s}_{L^q(B_r(z))}
\end{align}
for any $\tilde{u}\in \mc{H}$. Plugging $(\tilde{u}_n)_{n\in\N}$ into \eqref{conccompsummary}, assumption \eqref{Bedconccomp} entails $\Vert \mc{S}\tilde{u}_n\Vert_{L^s(D)} \to 0$ as $n\to\infty$. The desired result $\Vert \mc{S}\tilde{u}_n\Vert_{L^{\tilde{q}}(D)}$ as $n\to\infty$ for all $\tilde{q}\in (2,\bar q)$ then follows by H\"older interpolation. Since $\bar q\in (q,p^\ast+1)$ was arbitrary, Lemma~\ref{ConccompLemma} is proven.\qed

\section*{Acknowledgement}
We gratefully acknowledge financial support by the Deutsche Forschungsgemeinschaft (DFG) through CRC 1173.

\end{document}